\documentclass[11pt]{amsart}

\makeatletter
\@namedef{subjclassname@2020}{\emph{2020} Mathematics Subject Classification}
\makeatother

\usepackage[margin=1in]{geometry} 
\usepackage{amsmath, amsthm, amsfonts, amssymb} 
\usepackage{mathrsfs} 
\usepackage{enumerate} 
\usepackage{xcolor} 
\usepackage{bbm} 
\usepackage{hyperref} 

\usepackage{chngcntr}
\counterwithin{equation}{section} 

\usepackage{thmtools, thm-restate} 

\theoremstyle{theorem}
	
	\newtheorem{thm}{Theorem}[section]
	\newtheorem{lem}[thm]{Lemma}
	\newtheorem{prop}[thm]{Proposition}
	\newtheorem{cor}[thm]{Corollary}

\theoremstyle{definition}
	\newtheorem{defn}[thm]{Definition}
	\newtheorem{rem}[thm]{Remark}
	
	\newtheorem{eg}[thm]{Example}

\newcommand{\N}{\mathbb{N}}
\newcommand{\Z}{\mathbb{Z}}

\newcommand{\R}{\mathbb{R}}
\newcommand{\C}{\mathbb{C}}
\newcommand{\T}{\mathbb{T}}

\newcommand{\D}{\mathbb{D}}

\newcommand{\F}{\mathbb{F}}

\renewcommand{\P}{\mathcal{P}}

\newcommand{\eps}{\varepsilon}

\newcommand{\es}{\emptyset}

\newcommand{\lpf}{\normalfont\text{lpf}}

\renewcommand{\tilde}{\widetilde}
\renewcommand{\hat}{\widehat}

\DeclareMathOperator*{\E}{\text{\Large $\mathbb{E}$}}

\newcommand{\ind}{\mathbbm{1}}

\newcommand{\bfh}{\mathbf{h}}

\newcommand{\innprod}[2]{\left\langle #1, #2 \right\rangle}
\newcommand{\norm}[2]{\left\| #2 \right\|_{#1}}
\newcommand{\seminorm}[2]{{\left\vert\kern-0.25ex\left\vert\kern-0.25ex\left\vert #2 
    \right\vert\kern-0.25ex\right\vert\kern-0.25ex\right\vert}_{#1}}

\newcommand{\Hil}{\mathcal{H}}

\newcommand{\ddeg}{\text{d-}\deg}

\title{Polynomial patterns in subsets of large finite fields of low characteristic}

\author{Ethan Ackelsberg}
\address{School of Mathematics, Institute for Advanced Study, Princeton, NJ 08540}
\email{eackelsberg@ias.edu}

\author{Vitaly Bergelson}
\address{Department of Mathematics, Ohio State University, Columbus, OH 43210}
\email{vitaly@math.ohio-state.edu}

\date{\today}



\begin{document}

\maketitle


\begin{abstract}
	We prove a low characteristic counterpart to the main result in \cite{peluse},
	establishing power saving bounds for the polynomial Szemer\'{e}di theorem for certain families of polynomials.
	Namely, we show that if $P_1, \dots, P_m \in (\F_p[t])[y]$ satisfy an equidistribution condition,
	which is a natural variant of the independence condition in \cite{peluse} for our context,
	then there exists $\gamma > 0$ such that for any $q = p^k$ and any $A_0, A_1, \dots, A_m \subseteq \F_q$,
	\begin{align*}
		\left| \left\{ (x,y) \in \F_q^2 : x \in A_0, x + P_1(y) \in A_1, \dots, x + P_m(y) \in A_m \right\} \right| \\
		 = q^{-(m-1)} \prod_{i=0}^m{|A_i|}
		 + O_{q \to \infty; P_1, \dots, P_m} \left( |A_0|^{1/2} q^{3/2 - \gamma} \right).
	\end{align*}
	In particular, if $A \subseteq \F_q$ contains no
	pattern $\{x, x + P_1(y), \dots, x + P_m(y)\}$ of cardinality $m+1$, then
	\begin{align*}
		|A| \ll_{P_1, \dots, P_m} q^{1 - \gamma/ \left( m + \frac{1}{2} \right)}.
	\end{align*}
\end{abstract}




\section{Introduction}


\subsection{Background and motivation}

The goal of this paper is to establish a low characteristic analogue of the result on polynomial configurations in subsets of finite fields obtained in \cite{peluse}.

\begin{thm}[\cite{peluse}, Theorem 1.1] \label{thm: peluse}
	Let $P_1, \dots, P_m \in \Z[y]$ be linearly independent polynomials with $P_i(0) = 0$.
	There exist $C, \gamma > 0$ such that, for any finite field $\F_q$ of characteristic $p \ge C$
	and any set $A \subseteq \F_q$,
	\begin{align*}
		\left| \left\{ (x,y) \in \F_q^2 : x, x + P_1(y), \dots, x + P_m(y) \in A \right\} \right|
		 = \frac{|A|^{m+1}}{q^{m-1}} + O_{q \to \infty; P_1, \dots, P_m} \left( q^{2 - \gamma} \right).
	\end{align*}
\end{thm}

The notation $a(q) = O_{q \to \infty; u_1, \dots, u_m}(b(q))$ means that
for some constant $C = C(u_1, \dots, u_m) > 0$, one has $|a(q)| \le C |b(q)|$ for all large enough $q$.

Let $q = p^k$.
The setting of Theorem \ref{thm: peluse} is finite fields with the characteristic $p$ being sufficiently large.
We are interested in obtaining a result of a similar nature when $p$ is fixed and $k$ is large enough.
In Theorem \ref{thm: peluse}, the polynomials have coefficients in $\Z$.
In our characteristic $p$ setting (with $p$ fixed), the natural choice of coefficients is $\F_p[t]$.

Now, when dealing with polynomials in characteristic $p$, one encounters new phenomena
when considering polynomials with degree $d \ge p$.
The case when the polynomials $P_1, \dots, P_m$ are all of low degree $(d < p)$ is simpler than the general case
and can be handled with rather minor adjustments to existing methods;
see the discussion in Subsection \ref{sec: proof strategy} below.
The situation for high degree polynomials $(d \ge p)$ is substantially more complicated and requires new ideas.
To illustrate the distinction between low degree and high degree polynomials,
we give a brief overview of Weyl-type equidistribution results in characteristic $p$.
The notation and terminology will be useful in the sequel,
and the equidistribution results serve as important heuristics for the finitary results we prove in this paper.

We work with the ring $\F_p[t]$, which we view as an analogue of the integers $\Z$.
For $n = \sum_{j=0}^N{c_jt^j} \in \F_p[t]$ with $c_N \ne 0$, we define the absolute value $|n| = p^N$.
The absolute value extends to the field of rational function $\F_p(t) = \left\{ \frac{m}{n} : n, m \in \F_p[t], n \ne 0 \right\}$
by $\left| \frac{m}{n} \right| = \frac{|m|}{|n|}$.
With the metric induced by this absolute value, the completion of $\F_p(t)$ is the field
\begin{align*}
	\F_p((t^{-1})) = \left\{ \sum_{j=-\infty}^N{c_j t^j} : N \in \Z, c_j \in \F_p \right\},
\end{align*}
which we view as an analogue of the real numbers $\R$.

We say that $a : \F_p[t] \to \F_p((t^{-1}))$ is \emph{well-distributed} mod $\F_p[t]$ if
\begin{align*}
	\lim_{N \to \infty}{\frac{1}{|\Phi_N|} \sum_{n \in \Phi_N}{f(a(n))}} = \int_{\F_p((t^{-1}))/\F_p[t]}{f~dm}
\end{align*}
for every continuous function $f : \F_p((t^{-1}))/\F_p[t] \to \C$ and every F{\o}lner sequence\footnote{
A \emph{F{\o}lner sequence} in $\F_p[t]$ is a sequence $(\Phi_N)_{N \in \N}$ of finite subsets of $\F_p[t]$
such that, for any $n \in \F_p[t]$,
\begin{align*}
	\lim_{N \to \infty}{\frac{\left| (\Phi_N + n) \triangle \Phi_N \right|}{|\Phi_N|}} = 0.
\end{align*}
Examples include $\Phi_N = \left\{ c_{N-1}t^{N-1} + \dots + c_1t + c_0 : c_i \in \F_p \right\}$
(the set of all polynomials over $\F_p$ of degree $< N$)
and $\Phi'_N = \left\{ t^N + c_{N-1}t^{N-1} + \dots + c_1t + c_0 : c_i \in \F_p \right\} = \Phi_N + t^N$
(the set of all monic polynomials of degree $N$).}
$(\Phi_N)_{N \in \N}$ in $\F_p[t]$.

The distributional behavior of low degree polynomial $\F_p[t]$-sequences in the ``torus'' $\F_p((t^{-1}))/\F_p[t]$
is directly analagous to Weyl's equidistribution theorem on polynomial $\Z$-sequences in $\T$.
Say that an element $\alpha \in \F_p((t^{-1}))$ is \emph{irrational} if $\alpha \notin \F_p(t)$.

\begin{thm}[\cite{bl}, Theorem 0.2] \label{thm: low degree Weyl}
	Let $P(n) = \alpha_d n^d + \dots + \alpha_1 n + \alpha_0 \in \F_p((t^{-1}))[n]$ with degree $d < p$.
	If at least one of the coefficients $\alpha_1, \dots, \alpha_d$ is irrational,
	then $\left( P(n) \right)_{n \in \F_p[t]}$ is well-distributed mod $\F_p[t]$.
\end{thm}

The distributional behavior of high degree polynomials is more intricate.
For example, if $P(n) = n^p$, then one can find irrational $\alpha \in \F_p((t^{-1}))/\F_p[t]$ such that
$\left( P(n) \alpha \right)_{n \in \F_p[t]}$ is not well-distributed mod $\F_p[t]$.
Indeed, for any $\alpha$ of the form $\alpha = \beta^p$, the orbit closure $\overline{\left\{ P(n)\alpha : n \in \F_p[t] \right\}}$
is contained in the infinite index subgroup $\left\{ x^p : x \in \F_p((t^{-1}))/\F_p[t] \right\}$.

However, there is a special class of polynomials for which
a straightforward extension of Theorem \ref{thm: low degree Weyl} holds.

\begin{defn}
	A polynomial $P(y) \in \F_p((t^{-1}))[y]$ is called \emph{separable}
	if $P(y) = a_0 + \sum_{i=1}^k{a_i y^{r_i}}$ and $p \nmid r_i$ for $i \in \{1, \dots, k\}$.
\end{defn}

\begin{thm}[\cite{bl}, Corollary 0.5] \label{thm: Weyl separable}
	Let $P(n)$ be a separable polynomial with at least one irrational coefficient other than the constant term.
	Then $\left( P(n) \right)_{n \in \F_p[t]}$ is well-distributed mod $\F_p[t]$.
\end{thm}

In general, the distributional behavior of a high degree polynomial more closely resembles
the distributional behavior of polynomial sequences in higher-dimensional tori.
For this discussion, we need to introduce yet another special class of polynomials.

\begin{defn}
	A polynomial $\eta(y) \in \F_p((t^{-1}))[y]$ is \emph{additive} if $\eta(x+y) = \eta(x) + \eta(y)$
	for every $x, y \in \F_p((t^{-1}))$.
\end{defn}

Note that a polynomial $P(y) \in \F_p((t^{-1}))[y]$ is additive if and only if $P(y) = \sum_{i=0}^k{a_i y^{p^i}}$
for some $k \ge 0$ and some coefficients $a_0, \dots, a_k \in \F_p((t^{-1}))$.
(This is a simple consequence of properties of binomial coefficients mod $p$.)
Hence, there are high degree polynomials (such as the example $P(n) = n^p$ discussed above)
that behave like linear polynomials in the sense that they are additive.

This observation leads to another notion of degree that will be useful for us.
For any function $f$ defined on $\F_p((t^{-1}))$ and taking values in an abelian group,
define a differencing operator by $\partial_uf(x) = f(x+u) - f(x)$ for $x, u \in \F_p((t^{-1}))$.

\begin{defn}
	The \emph{derivational degree} (abbreviated d-$\deg$) of a polynomial $P(y) \in \F_p((t^{-1}))[y]$
	is the minimum $d \ge 0$ such that the $\partial_{u_1, \dots, u_{d+1}}P(y) = 0$
	for any $u_1, \dots, u_{d+1}, y \in \F_p((t^{-1}))$.
\end{defn}

Additive polynomials have $\ddeg$ equal to 1 by definition.
A monomial $y^r$ has $\ddeg{y^r} = \sum_{i=0}^k{a_i}$, where $a_i$ are the digits of the base $p$ expansion of $r$,
i.e. $r = \sum_{i=0}^k{a_i p^i}$, $a_i \in \{0, \dots, p-1\}$.
This can be seen by writing $y^p = \prod_{i=0}^k{\left( y^{a_i} \right)^{p^i}}$ as a product of monomials
with derivational degrees $a_0, \dots, a_k$
(composition with the additive polynomials $\eta_i(y) = y^{p^i}$ does not change the derivational degree).

Before formulating the general Weyl-type equidistribution theorem from \cite{bl},
we need to refine the notion of well-distribution discussed above.
A function $a : \F_p[t] \to \F_p((t^{-1}))$ is \emph{well-distributed mod $\F_p[t]$
in a subgroup $H \subseteq \F_p((t^{-1}))/\F_p[t]$} if
\begin{align*}
	\lim_{N \to \infty}{\frac{1}{|\Phi_N|} \sum_{n \in \Phi_N}{f(a(n))}} = \int_H{f~dm_H}
\end{align*}
for every continuous function $f : \F_p((t^{-1}))/\F_p[t] \to \C$ and every F{\o}lner sequence $(\Phi_N)_{N \in \N}$ in $\F_p[t]$.
For a subgroup $H$ and a finite set $F \subseteq \F_p((t^{-1}))/\F_p[t]$, we say that $a : \F_p[t] \to \F_p((t^{-1}))$ is
\emph{well-distributed in the components of $H + F$} if there exists $m \in \F_p[t] \setminus \{0\}$ such that,
for every $k \in \F_p[t]$, there exists $x \in F$ so that $\left( a(mn + k) \right)_{n \in \F_p[t]}$ is well-distributed in $H + x$.

\begin{thm}[\cite{bl}, Theorem 0.3] \label{thm: Weyl}
	An additive polynomial $\eta(y) \in (\F_p((t^{-1})))[y]$ is well distributed in the subgroup\footnote{In \cite{bl},
	the subgroup $\mathcal{F}(\eta)$ is called a \emph{$\Phi$-subtorus of level $\le \log_p{d}$}}
	$\overline{\eta(\F_p[t])} = \mathcal{F}(\eta) + \eta(K)$, where $K \subseteq \F_p((t^{-1}))/\F_p[t]$ is a finite subgroup.
	For any polynomial $P(y) = \alpha_0 + \sum_{i=1}^n{\eta_i(y^{r_i})}$, the orbit closure
	$\mathcal{O}(P) = \overline{P(\F_p[t])}$ is of the form $\mathcal{F}(P) + P(K)$,
	where $\mathcal{F}(P) = \sum_{i=1}^n{\mathcal{F}(\eta_i)}$ and $K$ is a finite subset of $\F_p[t]$,
	and $P(y)$ is well-distributed in the components $\mathcal{F}(P) + P(k)$, $k \in K$.
\end{thm}


\subsection{Main results}

\begin{defn}
	A family of sequences $a_1, \dots, a_m : \F_p[t] \to \F_p[t]$ is \emph{good for irrational equidistribution} if
	for any $\alpha_1, \dots, \alpha_m \in \F_p((t^{-1}))$, not all rational,
	$\left( \sum_{i=1}^m{a_i(n) \alpha_i} \right)_{n \in \F_p[t]}$ is well-distributed mod $\F_p[t]$.
\end{defn}

\begin{eg}
	Let $P_1, \dots, P_m \in \F_p[t]$ be separable polynomials.
	Then $\{P_1, \dots, P_m\}$ is good for irrational equidistribution if and only if $P_1, \dots, P_m$ are \emph{independent},
	i.e. every nontrivial linear combination of $P_1, \dots, P_m$ is non-constant.
	This follows from Theorem \ref{thm: Weyl separable} above.
\end{eg}

Before stating our main result, we formulate a related result for a single polynomial obtained
in a companion paper \cite{ab1}.
A special case of this result serves as the basis for our induction process and is instrumental in this paper.
It also serves as an illustration of the connection between infinitary equidistribution results and finitary combinatorial results.
To state the result concisely, we need some additional notation.
For a finite set $S$ and a function $f : S \to \C$, we write
\begin{align*}
	\E_{x \in S}{f(x)} = \frac{1}{|S|} \sum_{x \in S}{f(x)}.
\end{align*}
For $r \ge 1$, we define the $L^r$-norm on $S$ by
\begin{align*}
	\norm{L^r(S)}{f} := \left( \E_{x \in S}{\left| f(x) \right|^r} \right)^{1/r}.
\end{align*}
We denote the set of monic polynomials over $\F_p$ by $\F_p[t]^+$.
Given $Q(t) \in \F_p[t]^+$, we write $\F_p[t]_Q$ for the quotient ring $\F_p[t]/Q(t) \F_p[t]$.
If $Q$ factors into irreducibles as $Q = \prod_{i=1}^r{Q_i^{s_i}}$, we write $\lpf(Q) = \min_{1 \le i \le r}{|Q_i|}$.
Note that if $Q(t) \in \F_p[t]^+$ is irreducible, then $\F_p[t]_Q$ is isomorphic to the finite field $\F_q$ with $q = |Q|$,
and $\lpf(Q) = |Q|$.

\begin{thm}[\cite{ab1}, Theorem 1.18] \label{thm: Sarkozy}
	Let $P(y) \in \F_p[y]$.
	Let $\eta_1, \dots, \eta_n \in \F_p[y]$ be additive polynomials
	and $r_1, \dots, r_n \in \N$ distinct positive integers not divisible by $p$ so that
	$P(y) = \sum_{i=1}^n{\eta_i(y^{r_i})}$.
	The following are equivalent:
	\begin{enumerate}[(i)]
		\item	$P(y)$ is good for irrational equidistribution;
		\item	there exist additive polynomials $\zeta_1, \dots, \zeta_n \in \F_p[y]$ and $a \in \F_p^{\times}$ such that
			\begin{align*}
				\sum_{i=1}^n{(\eta_i \circ \zeta_i)(y)} = ay;
			\end{align*}
		\item	for any $Q(t) \in \F_p[t]^+$,
			\begin{align*}
				\sup_{\norm{L^2(\F_p[t]_Q)}{f} = 1}{\norm{L^2(\F_p[t]_Q)}
				 {\E_{y \in \F_p[t]_Q}{f(x+P(y))} - \E_{z \in \F_p[t]_Q}{f(z)}}}
				 = o_{\lpf(Q) \to \infty}(1);
			\end{align*}
		\item	there exist $C_1, C_2, \gamma > 0$ such that for any $Q(t) \in \F_p[t]^+$ with $\lpf(Q) \ge C_1$, one has
			\begin{align*}
				\sup_{\norm{L^2(\F_p[t]_Q)}{f} = 1}{\norm{L^2(\F_p[t]_Q)}
				{\E_{y \in \F_p[t]_Q}{f(x+P(y))} - \E_{z \in \F_p[t]_Q}{f(z)}}}
				 \le C_2 \cdot \lpf(Q)^{-\gamma};
			\end{align*}
		\item there exists $C > 0$ such that if $Q(t) \in \F_p[t]^+$ and $\lpf(Q) \ge C$, then $H_Q = \F_p[t]$,
			where $H_Q = \sum_{i=1}^n{H_{i,Q}}$, $H_{i,Q} = \eta_i(\F_p[t]_Q)$.
		\item	for any $\delta > 0$, there exists $N > 0$ such that if $Q(t) \in \F_p[t]^+$ has $\lpf(Q) \ge N$ and
			$A, B \subseteq \F_p[t]_Q$ are subsets with $|A| |B| \ge \delta |Q|^2$, then
			there exist $x, y \in \F_p[t]_Q$ such that $x \in A$ and $x + P(y) \in B$;
		\item	there exist $C_1, C_2, \gamma > 0$ such that for any $Q(t) \in \F_p[t]^+$ with $\lpf(Q) \ge C_1$, one has
			\begin{align*}
				\bigg| \left| \left\{ (x,y) \in \F_p[t]_Q^2 : x \in A, x + P(y) \in B \right\} \right| - |A| |B| \bigg|
				 \le C_2 |A|^{1/2} |B|^{1/2} |Q| \cdot \lpf(Q)^{-\gamma}.
			\end{align*}
	\end{enumerate}
\end{thm}

We also show in \cite{ab1} that, for general polynomials $P(y) \in (\F_p[t])[y]$, if $P(y)$ is good for irrational equidistribution,
then conditions (iii)-(vii) hold and are all equivalent; see \cite[Theorem 1.17]{ab1}.

Our main result in the present paper is upgrading condition (iv) from Theorem \ref{thm: Sarkozy}
to a result for averages involving multiple polynomials in the context of finite fields:

\begin{thm} \label{thm: asymptotic joint ergodicity}
	Suppose $\{P_1, \dots, P_m\} \subseteq (\F_p[t])[y]$ is good for irrational equidistribution.
	Then there exists $\gamma > 0$ such that for any $q = p^k$ and any $f_1, \dots, f_m : \F_q \to \D$,
	\begin{align*}
		\norm{L^2(\F_q)}{\E_{y \in \F_q}{\prod_{i=1}^m{f_i(x+P_i(y))}} - \prod_{i=1}^m{\E_{z \in \F_q}{f_i(z)}}}
		 \ll_{P_1, \dots, P_m} q^{-\gamma}.
	\end{align*}
\end{thm}

\begin{rem}
	(1) A version of Theorem \ref{thm: asymptotic joint ergodicity} can be formulated for the quotient rings $\F_p[t]_Q$.
	Extending our results to this more general setting is work in progress.
	
	(2) Although Theorem \ref{thm: asymptotic joint ergodicity} includes $m=1$ as a special case,
	it does not supersede the results in \cite{ab1}.
	Instead, we use the $m=1$ case as a blackbox from \cite{ab1} as the base case for an induction argument;
	see Subsection \ref{sec: proof strategy} below for an overview of the proof.
\end{rem}

From Theorem \ref{thm: asymptotic joint ergodicity},
we deduce a low characteristic (high degree) counterpart of Theorem \ref{thm: peluse}:

\begin{restatable}{cor}{PowerSaving} \label{cor: power saving}
	Suppose $\{P_1, \dots, P_m\} \subseteq (\F_p[t])[y]$ is good for irrational equidistribution.
	Then there exists $\gamma > 0$ such that for any $q = p^k$ and any $A_0, A_1, \dots, A_m \subseteq \F_q$,
	\begin{align*}
		\left| \left\{ (x,y) \in \F_q^2 : x \in A_0, x + P_1(y) \in A_1, \dots, x + P_m(y) \in A_m \right\} \right| \\
		 = q^{-(m-1)} \prod_{i=0}^m{|A_i|}
		 + O_{q \to \infty; P_1, \dots, P_m} \left( |A_0|^{1/2} q^{3/2 - \gamma} \right).
	\end{align*}
	In particular, if $A \subseteq \F_q$ contains no nontrivial\footnote{
	By a nontrivial pattern, we mean that the elements $x, x + P_1(y), \dots, x + P_m(y)$ are all distinct.}
	pattern $\{x, x + P_1(y), \dots, x + P_m(y)\}$, then
	\begin{align*}
		|A| \ll_{P_1, \dots, P_m} q^{1 - \gamma/ \left( m + \frac{1}{2} \right)}.
	\end{align*}
\end{restatable}

We prove Corollary \ref{cor: power saving} in Section \ref{sec: power saving}. \\

The conclusion of Corollary \ref{cor: power saving} includes two features when compared with Theorem \ref{thm: peluse}.
First, we allow for different sets $A_0, \dots, A_m$ rather than a single set $A$.
Second, we make no assumptions about the constant terms of the polynomials $P_1, \dots, P_m$.
These features are related to each other in the following way: if one takes $A_i = A_0 + P_i(0)$ for each $i = 1, \dots, m$,
then $x + P_i(y) \in A_i$ is equivalent to $x + \tilde{P}_i(y) \in A_0$, where $\tilde{P}_i = P_i - P_i(0)$
is a shift of $P_i$ having zero constant term.
So, the ability to take $m+1$ different sets $A_0, \dots, A_m$ is reflective of the fact that the result is invariant
under shifts of the polynomials $P_1, \dots, P_m$.

The independence of Corollary \ref{cor: power saving} of the behavior of the constant terms of the polynomials
$P_1, \dots, P_m$ stands in stark contrast with the polynomial Szemer\'{e}di theorem.
In \cite[Theorem 5.7]{blm}, it is shown that for any polynomials $P_1, \dots, P_m$ with zero constant terms
and any $\delta > 0$, there exists $N$ such that for any subset $A \subseteq \F_p[t]$
consisting of polynomials of degree $< N$ and with cardinality $|A| \ge \delta p^N$,
there exists a nontrivial configuration $\{x, x + P_1(y), \dots, x + P_m(y)\} \subseteq A$.
The role of zero constant term is to avoid ``local obstructions.''
Indeed, a necessary (though weaker) assumption in order
for the conclusion of the polynomial Szemer\'{e}di theorem to hold
is that the polynomials $P_1, \dots, P_m$ be \emph{jointly intersective},
meaning that for every $Q(t) \in \F_p[t]^+$, the polynomials $P_1, \dots, P_m$ have a common root mod $Q$.
To see this, one may consider the sets $A_N = \left\{ Qn : \deg{n} < N - \deg{Q} \right\}$.
In the finite field setting in Corollary \ref{cor: power saving}, we avoid issues coming from local obstructions
because any element $Q(t) \in \F_p[t]^+$ is invertible in $\F_q$ once $q$ is sufficiently large (depending on $Q$).
The absence of local obstructions is behind the phenomenon of \emph{asymptotic total ergodicity}
discussed in \cite{ab1}.

Since local obstructions are not a factor in our setting,
some discussion is in order on the role of irrational equidistribution
in Theorem \ref{thm: asymptotic joint ergodicity} and Corollary \ref{cor: power saving}.
Already in the case of a single polynomial ($m=1$), the fact that we may take different sets $A_0, A_1$
imposes an equidistribution condition; see items (vi) and (vii) in Theorem \ref{thm: Sarkozy} above.
For concreteness, we now give a specific example of a polynomial
for which the conclusion of Corollary \ref{cor: power saving} fails.
Let $P(y) = y^p - y$.
The polynomial $P$ has $p$ roots in $\F_q$.
In particular, every element of $\F_p \subseteq \F_q$ is a root of $P$.
Moreover, $P$ is an additive polynomial, so the image $H_q = P(\F_q)$ is a subgroup of index $p$ in $\F_q$
for every $q = p^k$.
Taking $A_0 = H_q$ and $A_1$ to be a nontrivial coset of $H_q$, it follows that
$C = \{(x,y) \in \F_q^2 : x \in A_0, x + P(y) \in A_1\} = \es$.
This disagrees with the prediction from Corollary \ref{cor: power saving} that $C$ should have cardinality
$|C| = p^{-2} q^2 + O_{q \to \infty; P} \left( q^{2 - \gamma} \right)$.


\subsection{Proof strategy} \label{sec: proof strategy}

The main technical ingredient in the proof of Theorem \ref{thm: asymptotic joint ergodicity} is the following general criterion,
which can be seen as a finitary version of the joint ergodicity criterion of Frantzikinakis \cite{fra} (see also \cite{bf}),
based on the work of Peluse \cite{peluse} and Peluse and Prendiville \cite{pp}:

\begin{restatable}{thm}{JointErgodicity} \label{thm: joint ergodicity}
	Let $G_1$ and $G_2$ be finite abelian groups, and let $a_1, \dots, a_m : G_2 \to G_1$.
	Suppose there exist $C > 0$, $s \in \N$, and $\alpha \in (0,1]$ such that
	\begin{enumerate}[(i)]
		\item	for any $l \in \{1, \dots, m\}$, any $f_1, \dots, f_l : G_1 \to \D$, and any $\chi_{l+1}, \dots, \chi_m \in \hat{G}_1$,
			\begin{align*}
				\norm{L^2(G_1)}{\E_{y \in G_2}{\prod_{i=1}^l{f_i(x + a_i(y))} \prod_{j=l+1}^m{\chi_j(a_j(y))}}}
				 \le C \norm{U^s(G_1)}{f_l}^{\alpha} + \delta_1,
			\end{align*}
			and
		\item	for any $\chi_1, \dots, \chi_m \in \hat{G}_1$, not all the trivial character,
			\begin{align*}
				\left| \E_{y \in G_2}{\prod_{i=1}^m{\chi_i(a_i(y))}} \right| \le \delta_2.
			\end{align*}
	\end{enumerate}
	Then there exist $\gamma_1, \gamma_2 > 0$ such that for any $f_1, \dots, f_m : G_1 \to \D$,
	\begin{align*}
		\norm{L^2(G_1)}{\E_{y \in G_2}{\prod_{i=1}^m{f_i(x + a_i(y))}} - \prod_{i=1}^m{\E_{z \in G_1}{f_i(z)}}}
		 \ll_{m, C, s, \alpha} \delta_1^{\gamma_1} + \delta_2^{\gamma_2}.
	\end{align*}
\end{restatable}

We prove Theorem \ref{thm: joint ergodicity} in Section \ref{sec: joint ergodicity criterion}.

It then remains to establish properties (i) and (ii) when $G_1 = G_2 = \F_q$ and $a_i = P_i \in (\F_p[t])[y]$
(with $\delta_j = q^{-\beta_j}$ for some $\beta_1, \beta_2 > 0$).

In the case that all of the polynomials $P_1, \dots, P_m$ have degree $d < p$,
property (i) follows from a standard PET induction argument; see, e.g. \cite{prendiville-pet, peluse}.
Property (ii) follows from elementary estimates of exponential sums over finite fields given in \cite{bbi};
see \cite[Lemma 3]{bbi} and the final inequality appearing in its proof at the top of page 713 in \cite{bbi}.

Proving properties (i) and (ii) in the high degree case is more challenging.
In order to establish property (i), we use a new variant of PET induction
that is based on the derivational degrees of the polynomials $P_1, \dots, P_m$.
This argument is carried out in Section \ref{sec: Gowers norm estimates}.
For condition (ii), we use a Furstenberg--S\'{a}rk\"{o}zy-type result proved in \cite{ab1};
see item (iv) of Theorem \ref{thm: Sarkozy} above.


\section{Joint ergodicity criterion over finite abelian groups} \label{sec: joint ergodicity criterion}

The goal of this section is to prove Theorem \ref{thm: joint ergodicity}, restated here:

\JointErgodicity*

\begin{rem}
	The conclusion of Theorem \ref{thm: joint ergodicity} implies conditions (i) and (ii) for appropriate constants.
	Suppose
	\begin{align*}
		\norm{L^2(G_1)}{\E_{y \in G_2}{\prod_{i=1}^m{f_i(x + a_i(y))}} - \prod_{i=1}^m{\E_{z \in G_1}{f_i(z)}}} \le \eps
	\end{align*}
	for every $f_1, \dots, f_m : G_1 \to \D$.
	Then by the triangle inequality,
	\begin{align*}
		\norm{L^2(G_1)}{\E_{y \in G_2}{\prod_{i=1}^m{f_i(x + a_i(y))}}} \le \prod_{i=1}^m{\norm{U^1(G_1)}{f_i}} + \eps
		 \le \min_{1 \le i \le m}{\norm{U^1(G_1)}{f_i}} + \eps,
	\end{align*}
	so (i) holds with $C = 1$, $s = 1$, $\alpha = 1$, and $\delta_1 = \eps$.
	Specializing to $f_i = \chi_i \in \hat{G}_1$ immediately gives (ii) with $\delta_2 = \eps$.
\end{rem}


\subsection{Outline of the proof} \label{sec: outline}

In order to carry out a proof of Theorem \ref{thm: joint ergodicity}, we will prove the following statement,
which gives Theorem \ref{thm: joint ergodicity} in the case $l = m$:

\begin{prop} \label{prop: inductive version}
	Suppose $a_1, \dots, a_m : G_2 \to G_1$ satisy (i) and (ii).
	Let $l \in \{0, \dots, m\}$.
	There exist $\gamma_{l, 1}, \gamma_{l,2} > 0$ such that
	for any $f_1, \dots, f_l : G_1 \to \D$ and any $\chi_{l+1}, \dots, \chi_m \in \hat{G}_1$,
	\begin{align*}
		\norm{L^2(G_1)}{\E_{y \in G_2}{\prod_{i=1}^l{f_i(x + a_i(y))} \prod_{j=l+1}^m{\chi_j(a_j(y))}}
		 - \ind_{\chi=1} \prod_{i=1}^l{\E_{z \in G_1}{f_i(z)}}}
		 \ll_{l,C,s,\alpha} \delta_1^{\gamma_{l,1}} + \delta_2^{\gamma_{l,2}}.
	\end{align*}
\end{prop}

The base case $l=0$ follows from condition (ii).

We now briefly describe how to prove Proposition \ref{prop: inductive version} for $l \in \{1, \dots, m\}$ given the result for $l-1$.
By the triangle inequality, it suffices to bound the quantity
\begin{align} \label{eq: norm to estimate}
	\norm{L^2(G_1)}{\E_{y \in G_2}{\prod_{i=1}^l{f_i(x + a_i(y))} \prod_{j=l+1}^m{\chi_j(a_j(y))}}}
\end{align}
under the additional assumption that $\E_{z \in G_1}{f_l(z)} = 0$.
Suppose the quantity in \eqref{eq: norm to estimate} is large.
We first show (Lemma \ref{lem: dual function}) that this remains true upon replacing $f_l$ by a function of the form
\begin{align*}
	\tilde{f}_l(x) = \E_{y \in G}{f_0(x - a_l(y)) \prod_{i=1}^{l-1}{f_i(x + (a_i - a_l)(y))} \prod_{j=l+1}^m{\chi_j((a_j - a_l)(y))}}
\end{align*}
with $\norm{U^1(G_1)}{\tilde{f}_l}$ very small
(with the bound coming from the $l-1$ case of Proposition \ref{prop: inductive version}).
The condition (i) then implies that $\norm{U^s(G_1)}{\tilde{f}_l}$ is large.
The goal is then to show that the quantity $\norm{U^{s-1}(G_1)}{\tilde{f}_l}$ is also large
and to iterate the degree lowering process until we arrive at a contradiction
by showing that $\norm{U^1(G_1)}{\tilde{f}_l}$ must be large.

The degree lowering process is roughly as follows.
Largeness of the $U^s$-norm implies that (Proposition \ref{prop: U^s character correlations})
\begin{align*}
	\E_{\bfh \in G_1^{s-2}}{\left| \E_{x \in G_1}{\Delta_{\bfh} \tilde{f}_l(x) \chi_{\bfh}(x)} \right|}
\end{align*}
is large for some family $\left( \chi_{\bfh} \right)_{\bfh \in G^{s-2}}$ of characters on $G$.
We then swap the order of averaging in the definition of $\tilde{f}_l$ with the differencing operator
and deduce with the help of Lemma \ref{lem: difference interchange} that
\begin{align*}
	\E_{\bfh^0, \bfh^1 \in G_1^{s-2}}{\left| \E_{x \in G_1}{\E_{y \in G_2}{\Delta_{\bfh^0 - \bfh^1} f_0(x-a_l(y))
	 \prod_{i=1}^{l-1}{\Delta_{\bfh^0 - \bfh^1} f_i(x + (a_i-a_l)(y))} \chi_{\bfh^0, \bfh^1}(x)}} \right|}
\end{align*}
is large, where
\begin{align*}
		\chi_{\bfh^0, \bfh^1} = \prod_{\omega \in \{0,1\}^{s-2}}{C^{|\omega|} \chi_{\bfh^{\omega}}} \qquad \text{and} \qquad
		\bfh^{\omega} = \left( h^{\omega_1}_1, \dots, h^{\omega_{s-2}}_{s-2} \right).
	\end{align*}
Replacing $x$ by $x + a_l(y)$ and using the induction hypothesis, this implies
\begin{align*}
	\frac{\left| \left\{ (\bfh^0, \bfh^1) \in G_1^{2(s-2)} : \chi_{\bfh^0, \bfh^1} = 1 \right\} \right|}{|G_1|^{2(s-2)}}
\end{align*}
is large.
From the definition of $\chi_{\bfh^0, \bfh^1}$, this allows us to write the map $\bfh \mapsto \chi_{\bfh}$
as a product of functions of fewer variables (at least on a large subset of the domain $G_1^{s-2}$).
Utilizing this extra structure of the family of characters $\left( \chi_{\bfh} \right)_{\bfh \in G_1^{s-2}}$,
we are able to conclude (by applying Lemma \ref{lem: low rank}) that $\norm{U^{s-1}(G_1)}{\tilde{f}_l}$ is large.

For a full proof of Proposition \ref{prop: inductive version},
we need to give precise meaning to ``large'' and ``small'' in the preceding discussion
and to obtain quantitative control of the losses in each step of the argument.
We gather the required technical lemmas in Subsections \ref{sec: multilin}--\ref{sec: degree lowering}
and put everything together to prove Proposition \ref{prop: inductive version} in Subsection \ref{sec: proof}.


\subsection{Reductions using multlinearity} \label{sec: multilin}

\begin{lem} \label{lem: multilinear}
	Let $G_1$ and $G_2$ be finite abelian groups.
	Let $a_1, \dots, a_m : G_2 \to G_1$.
	Let $l \in \{1, \dots, m\}$, $f_1, \dots, f_l : G_1 \to \D$, and $\chi_{l+1}, \dots, \chi_m \in \hat{G}_1$.
	Suppose
	\begin{align*}
		\norm{L^2(G_1)}{\E_{y \in G_2}{\prod_{i=1}^l{f_i(x + a_i(y))} \prod_{j=l+1}^m{\chi_j(a_j(y))}}} \ge \delta.
	\end{align*}
	For each $i = 1, \dots, l-1$, write
	\begin{align*}
		f_i = \left( f_i - \E_{z \in G_1}{f_i(z)} \right) + \E_{z \in G_1}{f_i(z)} = f_{i,0} + f_{i,1}.
	\end{align*}
	Then for some $\omega \in \{0,1\}^{l-1}$,
	\begin{align*}
		\norm{L^2(G_1)}{\E_{y \in G_2}{\prod_{i=1}^{l-1}{f_{i, \omega_i}(x + a_i(y))} f_l(x + a_l(y))
		 \prod_{j=l+1}^m{\chi_j(a_j(y))}}} \ge \frac{\delta}{2^{l-1}}.
	\end{align*}
\end{lem}
\begin{proof}
	This is an easy application of the triangle inequality.
\end{proof}


\subsection{Replacing $f_l$ be a structured function}

Given a seminorm $\seminorm{}{\cdot}$, we define the \emph{dual seminorm} by
\begin{align*}
	\seminorm{}{f}^* = \sup_{\seminorm{}{g} \le 1}{\left| \innprod{f}{g} \right|}.
\end{align*}

\begin{lem} \label{lem: seminorm decomp}
	Let $\seminorm{}{\cdot}$ be a seminorm on $\C^G$.
	For any $f : G \to \C$ and any $c > 0$, there exist $f_s, f_u : G \to \C$ such that $f = f_s + f_u$ with
	\begin{align*}
		\seminorm{}{f_s}^* \le c^{-1} \norm{L^2(G)}{f} \qquad \text{and} \qquad \seminorm{}{f_u} \le c \norm{L^2(G)}{f}.
	\end{align*}
\end{lem}
\begin{proof}
	If $\seminorm{}{\cdot}$ is in fact a norm, then this result is shown in \cite[Proposition 3.6]{gowers}
	and the discussion immediately afterwards.
	For the extension to the seminorm case, see \cite[Lemma 7.1]{prendiville}.
\end{proof}

\begin{lem} \label{lem: dual function}
	Let $G_1$ and $G_2$ be finite abelian groups, and let $a_1, \dots, a_m : G_2 \to G_1$.
	Let $l \in \{1, \dots, m\}$.
	Suppose that for any $f_1, \dots, f_{l-1} : G_1 \to \D$ and any $\chi_l, \dots, \chi_m \in \hat{G}_1$,
	\begin{align*}
		\norm{L^2(G_1)}{\E_{y \in G_2}{\prod_{i=1}^{l-1}{f_i(x + a_i(y))} \prod_{j=l}^m{\chi_j(a_j(y))}}
		 - \ind_{\chi=1} \prod_{i=1}^{l-1}{\E_{z \in G_1}{f_i(z)}}} \le \eps.
	\end{align*}
	
	Let $f_1, \dots, f_l : G_1 \to \D$ and $\chi_{l+1}, \dots, \chi_m \in \hat{G}_1$.
	Assume each of the functions $f_1, \dots, f_{l-1}$ is either constant or has mean zero.
	If 
	\begin{align*}
		\norm{L^2(G_1)}{\E_{y \in G_2}{\prod_{i=1}^l{f_i(x + a_i(y))} \prod_{j=l+1}^m{\chi_j(a_j(y))}}} \ge \delta,
	\end{align*}
	then there exists $f_0 : G_1 \to \D$ such that
	\begin{align*}
		\norm{L^2(G_1)}{\E_{y \in G_2}{\prod_{i=1}^{l-1}{f_i(x + a_i(y))} \tilde{f}_l(x+a_l(y)) \prod_{j=l+1}^m{\chi_j(a_j(y))}}}
		 \ge \frac{\delta^3}{4}.
	\end{align*}
	where
	\begin{align*}
		\tilde{f}_l(x) = \E_{y \in G_2}{f_0(x - a_l(y)) \prod_{i=1}^{l-1}{f_i(x + (a_i - a_l)(y))} \prod_{j=l+1}^m{\chi_j((a_j - a_l)(y))}}.
	\end{align*}
	Moreover, if $\E_{x \in G_1}{f_l(x)} = 0$, we may choose $f_0$ so that
	\begin{align*}
		\left| \E_{x \in G_1}{\tilde{f}_l(x)} \right| \le \eps.
	\end{align*}
\end{lem}
\begin{proof}
	Let
	\begin{align*}
		f_0(x) = \E_{y \in G_2}{\prod_{i=1}^l{\overline{f_i(x + a_i(y))}} \prod_{j=l+1}^m{\overline{\chi_j(a_j(y))}}}
	\end{align*}
	so that
	\begin{align*}
		\E_{x \in G_1}{\E_{y \in G_2}{f_0(x) \prod_{i=1}^l{f_i(x + a_i(y))} \prod_{j=l+1}^m{\chi_j(a_j(y))}}}
		 = \norm{L^2(G_1)}{\E_{y \in G_2}{\prod_{i=1}^l{f_i(x + a_i(y))} \prod_{j=l+1}^m{\chi_j(a_j(y))}}}^2 \ge \delta^2.
	\end{align*}
	
	Decompose $f_l = f_s + f_u$ according to Lemma \ref{lem: seminorm decomp} for the seminorm
	\begin{align*}
		\seminorm{}{f}
		 = \norm{L^2(G_1)}{\E_{y \in G_2}{\prod_{i=1}^{l-1}{f_i(x + a_i(y))} f(x+a_l(y)) \prod_{j=l+1}^m{\chi_j(a_j(y))}}}
	\end{align*}
	Then by the triangle inequality and the Cauchy--Schwarz inequality,
	\begin{align*}
		\left| \innprod{f_s}{\tilde{f}_l} \right|
		 = \left| \E_{x \in G_1}{\E_{y \in G_2}{f_0(x) \prod_{i=1}^{l-1}{f_i(x + a_i(y))}
		 f_s(x + a_l(y)) \prod_{j=l+1}^m{\chi_j(a_j(y))}}} \right|
		 \ge \delta(\delta - c).
	\end{align*}
	On the other hand,
	\begin{align*}
		\left| \innprod{f_s}{\tilde{f}_l} \right| \le \seminorm{}{f_s}^* \seminorm{}{\tilde{f}_l}
		 \le c^{-1} \seminorm{}{\tilde{f}_l}.
	\end{align*}
	Therefore,
	\begin{align*}
		\seminorm{}{\tilde{f}_l} \ge c \delta (\delta - c).
	\end{align*}
	Taking $c = \frac{\delta}{2}$ yields the desired bound. \\
	
	We now compute $\left| \E_{x \in G_1}{\tilde{f}_l(x)} \right|$ directly from the definition of $\tilde{f}_l$:
	\begin{align*}
		\left| \E_{x \in G_1}{\tilde{f}_l(x)} \right|
		 = \left| \E_{x \in G_1}{\E_{y \in G_2}{f_0(x) \prod_{i=1}^{l-1}{f_i(x + a_i(y))} \prod_{j=l+1}^m{\chi_j(a_j(y))}}} \right|.
	\end{align*}
	If at least one of the functions $f_1, \dots, f_{l-1}$ has mean zero,
	then by the Cauchy--Schwarz inequality and the hypothesis,
	\begin{align*}
		\left| \E_{x \in G_1}{\E_{y \in G_2}{f_0(x) \prod_{i=1}^{l-1}{f_i(x + a_i(y))}
		 \prod_{j=l+1}^m{\chi_j(a_j(y))}}} \right|
		 \le \norm{L^2(G_1)}{\E_{y \in G_2}{\prod_{i=1}^{l-1}{f_i(x + a_i(y))}
		 \prod_{j=l+1}^m{\chi_j(a_j(y))}}}
		 \le \eps.
	\end{align*}
	If instead each $f_i$ is constant (say, equal to $c_i$), then
	\begin{align*}
		\E_{x \in G_1}{\tilde{f}_l(x)} = \E_{x \in G_1}{f_0(x)} \prod_{i=1}^{l-1}{c_i} \E_{y \in G_2}{\prod_{j=l+1}^m{\chi_j(a_j(y))}},
	\end{align*}
	and
	\begin{align*}
		\E_{x \in G_1}{f_0(x)} = \prod_{i=1}^{l-1}{\overline{c}_i} \E_{x \in G_1}{\E_{y \in G_2}{\overline{f}_l(x+a_l(y))
		 \prod_{j=l+1}^m{\overline{\chi_j(a_j(y))}}}}.
	\end{align*}
	Substituting $x+a_l(y)$ for $x$, it follows that if $\E_{x \in G_1}{f_l(x)} = 0$, then $\E_{x \in G_1}{\tilde{f}_l(x)} = 0$.
\end{proof}


\subsection{Gowers norms and correlations with characters}

\begin{prop} \label{prop: U^s character correlations}
	Let $G$ be a finite abelian group, and let $f : G \to \D$.
	There is a family $\left( \chi_{\bfh} \right)_{\bfh \in G^s}$ of characters on $G$ such that
	\begin{align*}
		\E_{\bfh \in G^s}{\left| \E_{x \in G}{\Delta_{\bfh} f(x) \chi_{\bfh}(x)} \right|} \ge \norm{U^{s+2}(G)}{f}^{2^{s+2}}.
	\end{align*}
\end{prop}
\begin{proof}
	By definition,
	\begin{align*}
		\norm{U^{s+2}(G)}{f}^{2^{s+2}} = \E_{\bfh \in G^s}{\norm{U^2(G)}{\Delta_{\bfh}f}^4}.
	\end{align*}
	Now, for each $\bfh \in G^s$, we have
	\begin{align*}
		\norm{U^2(G)}{\Delta_{\bfh}f}^4 = \sum_{\chi \in \hat{G}}{\left| \hat{\Delta_{\bfh}f}(\chi) \right|^4}
		 \le \sup_{\chi \in \hat{G}}{\left| \hat{\Delta_{\bfh}f}(\chi) \right|^2} \cdot \norm{L^2(G)}{f}^2,
	\end{align*}
	so we may choose $\chi_{\bfh} \in \hat{G}$ satisfying
	\begin{align*}
		\left| \E_{x \in G}{\Delta_{\bfh} f(x) \chi_{\bfh}(x)} \right|^2 \ge \norm{U^2(G)}{\Delta_{\bfh}f}^4.
	\end{align*}
	Using the fact that $f$ is 1-bounded, we then have
	\begin{align*}
		\E_{\bfh \in G^s}{\left| \E_{x \in G}{\Delta_{\bfh} f(x) \chi_{\bfh}(x)} \right|}
		 \ge \E_{\bfh \in G^s}{\left| \E_{x \in G}{\Delta_{\bfh} f(x) \chi_{\bfh}(x)} \right|^2}
		 \ge \E_{\bfh \in G^s}{\norm{U^2(G)}{\Delta_{\bfh}f}^4}
		 = \norm{U^{s+2}(G)}{f}^{2^{s+2}}.
	\end{align*}
\end{proof}

If the map $\bfh \mapsto \chi_{\bfh}$ is of low rank, then the $U^{s+1}$ norm must also be large:

\begin{lem} \label{lem: low rank}
	Let $G$ be a finite abelian group, and let $f : G \to \D$.
	Let $\varphi_1, \dots, \varphi_m : G^{s-1} \to \hat{G}$, $m \le s$, and define
	\begin{align*}
		\chi_{\bfh} = \prod_{i=1}^m{\varphi_i \left( (h_j)_{j \ne i} \right)}
	\end{align*}
	for $\bfh \in G^s$.
	Then
	\begin{align*}
		\E_{\bfh \in G^s}{\left| \E_{x \in G}{\Delta_{\bfh}f(x) \chi_{\bfh}(x)} \right|} \le \norm{U^{s+1}(G)}{f}^{2^{s-m}}.
	\end{align*}
\end{lem}
\begin{proof}
	For $m= 0$,
	\begin{align*}
		\E_{\bfh \in G^s}{\left| \E_{x \in G}{\Delta_{\bfh}f(x)} \right|}
		 \le \left( \E_{\bfh \in G^s}{\left| \E_{x \in G}{\Delta_{\bfh}f(x)} \right|^2} \right)^{1/2}
		 = \norm{U^{s+1}(G)}{f}^{2^s}.
	\end{align*} \\
	
	Now let $m \ge 1$.
	Letting $\psi(\bfh)$ be the conjugate of the phase of the absolute value on the left hand side, we have
	\begin{align*}
		\E_{\bfh \in G^s}&{\left| \E_{x \in G}{\Delta_{\bfh}f(x) \chi_{\bfh}(x)} \right|} \\
		 & = \E_{h_2, \dots, h_s \in G}{\E_{x \in G}{\Delta_{h_2, \dots, h_s}\overline{f(x)} \varphi_1(h_2, \dots, h_s)(x)
		 \E_{h_1 \in G}{\Delta_{h_2, \dots, h_s}f(x+h_1)
		 \prod_{i=2}^m{\varphi_i \left( (h_j)_{j \ne i} \right)(x)} \psi(h_1, \dots, h_s)}}} \\
		 & \le \left( \E_{h_2, \dots, h_s \in G}{\E_{x \in G}
		 {\left| \E_{h_1 \in G}{\Delta_{h_2, \dots, h_s}f(x+h_1)
		 \prod_{i=2}^m{\varphi_i \left( (h_j)_{j \ne i} \right)(x)} \psi(h_1, \dots, h_s)} \right|^2}} \right)^{1/2} \\
		 & \le \left( \E_{h_2, \dots, h_s \in G}{\E_{h_1, h'_1 \in G}{\left| \E_{x \in G}{\Delta_{h_1-h'_1, h_2, \dots, h_s}f(x)
		 \prod_{i=2}^m{\varphi_i \left( (h_j)_{j \ne i} \right)(x)
		 \overline{\varphi_i \left( (h'_j)_{j \ne i} \right)(x)}}} \right|}} \right)^{1/2} \\
		 & \le \left( \sup_{h'_1 \in G}{\E_{h_1, h_2, \dots, h_s \in G}{\left| \E_{x \in G}{\Delta_{h_1-h'_1, h_2, \dots, h_s}f(x)
		 \prod_{i=2}^m{\varphi_i \left( (h_j)_{j \ne i} \right)(x)
		 \overline{\varphi_i \left( (h'_j)_{j \ne i} \right)(x)}}} \right|}} \right)^{1/2}.
	\end{align*}
	By induction, this quantity is bounded above by
	\begin{align*}
		\left( \norm{U^{s+1}(G)}{f}^{2^{s-(m-1)}} \right)^{1/2}
		 = \norm{U^{s+1}(G)}{f}^{2^{s-m}}.
	\end{align*}
\end{proof}


\subsection{Difference interchange lemma}

A version of the following lemma appears as the ``dual-difference interchange'' lemma
in \cite[Lemma 6.3]{pp-arxiv} and \cite[Lemma 6.3]{prendiville}:

\begin{lem} \label{lem: difference interchange}
	Let $G_1$ and $G_2$ be finite abelian groups.
	Let $f : G_1 \times G_2 \to \D$, and define $F(x) = \E_{y \in G_2}{f(x,y)}$.
	For any $s \ge 0$, any collection $\left( \chi_{\bfh} \right)_{\bfh \in G_1^s}$ of characters on $G_1$,
	and any subset $\Hil \subseteq G_1^s$,
	\begin{align} \label{eq: difference interchange}
		\left( \frac{1}{|G_1|^s} \sum_{\bfh \in \Hil}{\left| \E_{x \in G_1}{\Delta_{\bfh} F(x) \chi_{\bfh}(x)} \right|} \right)^{2^s}
		 \le \frac{1}{|G_1|^{2s}} \sum_{\bfh^0, \bfh^1 \in \Hil}{\left| \E_{x \in G_1}{\E_{y \in G_2}
		 {\Delta^{(1)}_{\bfh^0 - \bfh^1} f(x,y) \chi_{\bfh^0, \bfh^1}(x)}} \right|},
	\end{align}
	where
	\begin{align*}
		\chi_{\bfh^0, \bfh^1} = \prod_{\omega \in \{0,1\}^s}{C^{|\omega|} \chi_{\bfh^{\omega}}} \qquad \text{and} \qquad
		\bfh^{\omega} = \left( h^{\omega_1}_1, \dots, h^{\omega_s}_s \right).
	\end{align*}
\end{lem}
\begin{proof}
	We prove the inequality \eqref{eq: difference interchange} by induction.
	
	When $s=0$, both sides of \eqref{eq: difference interchange} are equal to
	\begin{align*}
		\left| \E_{x \in G_1}{F(x) \chi(x)} \right| = \left| \E_{x \in G_1}{\E_{y \in G_2}{f(x,y) \chi(x)}} \right|.
	\end{align*}
	
	Suppose $s \ge 1$.
	Note that for $h \in G_1$,
	\begin{align*}
		\Delta_hF(x) = \E_{y,y' \in G_2}{f(x + h, y) \overline{f(x,y')}}.
	\end{align*}
	We break up the sum over $\Hil$ into an iterated sum (we write $\Hil^h = \left\{ \bfh \in G^{s-1} : (\bfh, h) \in \Hil \right\}$)
	and apply the induction hypothesis with the functions $f_h : G_1 \times G_2^2 \to \D$ given by
	$f_h(x,y,y') = f(x+h, y) \overline{f(x,y')}$:
	\begin{align*}
		\left( \frac{1}{|G_1|^s} \sum_{\bfh \in \Hil}{\left| \E_{x \in G_1}{\Delta_{\bfh} F(x) \chi_{\bfh}(x)} \right|} \right)^{2^s}
		 & = \left( \E_{h \in G_1}{\frac{1}{|G_1|^{s-1}} \sum_{\bfh \in \Hil^h}{\left| \E_{x \in G_1}{\Delta_{\bfh} \Delta_h F(x)
		 \chi_{\bfh, h}(x)} \right|}} \right)^{2^s} \\
		 & \le \left( \E_{h \in G_1}{\left( \frac{1}{|G_1|^{s-1}} \sum_{\bfh \in \Hil^h}{\left| \E_{x \in G_1}{\Delta_{\bfh} \Delta_h F(x)
		 \chi_{\bfh, h}(x)} \right|} \right)^{2^{s-1}}} \right)^2 \\
		 & \le \left( \E_{h \in G_1}{\frac{1}{|G_1|^{2(s-1)}} \sum_{\bfh^0, \bfh^1 \in \Hil^h}{\left| \E_{x \in G_1}{\E_{y,y' \in G_2}{
		 \Delta^{(1)}_{\bfh^0 - \bfh^1} f_h(x,y,y') \chi_{\bfh^0, \bfh^1, h}(x)}} \right|}} \right)^2.
	\end{align*}
	Let $\psi \left( \bfh^0, \bfh^1, h \right)$ be the conjugate of the phase of the absolute value.
	Then we move the average over $h$ to the inside and apply Cauchy--Schwarz:
	\begin{align*}
		& \left( \frac{1}{|G_1|^s} \sum_{\bfh \in \Hil}{\left| \E_{x \in G_1}{\Delta_{\bfh} F(x) \chi_{\bfh}(x)} \right|} \right)^{2^s} \\
		 &~= \left( \E_{\bfh^0, \bfh^1 \in G_1^{s-1}}{\E_{x \in G_1}{\E_{y, y' \in G_2}{\frac{1}{|G_1|} \sum_{(\bfh^i, h) \in \Hil}
		 {\Delta^{(1)}_{\bfh^0 - \bfh^1} f_h(x,y,y') \chi_{\bfh^0, \bfh^1, h}(x) \psi \left( \bfh^0, \bfh^1, h \right)}}}} \right)^2 \\
		 &~\le \E_{\bfh^0, \bfh^1 \in G_1^{s-1}}{\E_{x \in G_1}{\E_{y \in G_2}{\left| \frac{1}{|G_1|} \sum_{(\bfh^i, h) \in \Hil}
		 {\Delta^{(1)}_{\bfh^0 - \bfh^1} f(x + h,y) \chi_{\bfh^0, \bfh^1,h}(x) \psi \left( \bfh^0, \bfh^1, h \right)} \right|^2}}},
	\end{align*}
	since
	\begin{align*}
		\E_{\bfh^0, \bfh^1 \in G_1^{s-1}}{\E_{x \in G_1}{\E_{y' \in G_2}{\left| \Delta^{(1)}_{\bfh^0 - \bfh^1} f(x,y') \right|^2}}} \le 1.
	\end{align*}
	Expanding out the squared term,
	\begin{align*}
		& \left( \E_{\bfh \in G_1^s}{\left| \E_{x \in G_1}{\Delta_{\bfh} F(x) \chi_{\bfh}(x)} \right|} \right)^{2^s} \\
		 &~\le \E_{\bfh^0, \bfh^1 \in G_1^{s-1}}{\E_{x \in G_1}{\E_{y \in G_2}{\frac{1}{|G_1|^2} \sum_{(\bfh^i,h_j) \in \Hil}
		 {\Delta^{(1)}_{\bfh^0 - \bfh^1} f(x + h_0,y) \overline{f(x + h_1, y)}}}}} \\
		 & \qquad \qquad {{{{\left( \chi_{\bfh^0, \bfh^1,h_0} \overline{\chi_{\bfh^0, \bfh^1, h_1}} \right)(x)
		 \psi \left( \bfh^0, \bfh^1, h_0 \right) \overline{\psi \left( \bfh^0, \bfh^1, h_1 \right)}}}}} \\
		 &~\le \E_{\bfh^0, \bfh^1 \in G_1^{s-1}}{\frac{1}{|G_1|^2} \sum_{(\bfh^i, h_j) \in \Hil}
		 {\left| \Delta^{(1)}_{\bfh^0 - \bfh^1} \Delta^{(1)}_{h_0 - h_1} f(x,y) \chi_{\bfh^0, h_0; \bfh^1, h_1}(x) \right|}}.
	\end{align*}
\end{proof}


\subsection{Degree lowering} \label{sec: degree lowering}

\begin{lem} \label{lem: degree lowering}
	Let $G_1$ and $G_2$ be finite abelian groups, and let $a_1, \dots, a_m : G_2 \to G_1$.
	Let $l \in \{1, \dots, m\}$.
	Suppose that for any $f_1, \dots, f_{l-1} : G_1 \to \D$ and any $\chi_l, \dots, \chi_m \in \hat{G}_1$,
	\begin{align*}
		\norm{L^2(G_1)}{\E_{y \in G_2}{\prod_{i=1}^{l-1}{f_i(x + a_i(y))} \prod_{j=l}^m{\chi_j(a_j(y))}}
		 - \ind_{\chi=1} \prod_{i=1}^{l-1}{\E_{z \in G}{f_i(z)}}} \le \eps.
	\end{align*}
	
	Let $f_0, f_1, \dots, f_l : G_1 \to \D$ and $\chi_{l+1}, \dots, \chi_m \in \hat{G}_1$, and put
	\begin{align*}
		\tilde{f}_l(x) = \E_{y \in G_2}{f_0(x - a_l(y)) \prod_{i=1}^{l-1}{f_i(x + (a_i - a_l)(y))} \prod_{j=l+1}^m{\chi_j((a_j - a_l)(y))}}.
	\end{align*}
	If
	\begin{align*}
		\norm{U^{s+2}(G_1)}{\tilde{f}_l} \ge \delta,
	\end{align*}
	then
	\begin{align*}
		\norm{U^{s+1}(G_1)}{\tilde{f}_l}
		 \ge \left( \frac{\delta^{2^{2s+4}}}{2^{2^{s+2}}} - \eps^2 \right) \cdot \frac{\delta^{2^{s+2}}}{2}.
	\end{align*}
\end{lem}
\begin{proof}
	By Proposition \ref{prop: U^s character correlations}
	\begin{align*}
		\E_{\bfh \in G_1^s}{\left| \E_{x \in G_1}{\Delta_{\bfh}\tilde{f}_l(x) \chi_{\bfh}(x)} \right|} \ge \delta^{2^{s+2}}
	\end{align*}
	for some family $\left( \chi_{\bfh} \right)_{\bfh \in G_1^s}$ of characters on $G_1$.
	Therefore, there exists a subset $\Hil \subseteq G_1^s$ of size $|\Hil| \ge \frac{\delta^{2^{s+2}}}{2} |G_1|^s$ such that
	\begin{align*}
		\left| \E_{x \in G_1}{\Delta_{\bfh} \tilde{f}_l(x) \chi_{\bfh}(x)} \right| \ge \frac{\delta^{2^{s+2}}}{2}
	\end{align*}
	for every $\bfh \in \Hil$.
	
	Applying Lemma \ref{lem: difference interchange}, it follows that
	\begin{align*}
		\frac{1}{|G_1|^{2s}} \sum_{\bfh^0, \bfh^1 \in \Hil}{\left| \E_{x \in G_1}{\E_{y \in G_2}
		{\Delta_{\bfh^0 - \bfh^1}f_0(x - a_l(y)) \prod_{i=1}^{l-1}{\Delta_{\bfh^0 - \bfh^1}f_i(x + (a_i-a_l)(y))}
		\prod_{j=l+1}^m{\tilde{\chi}_j(a_j(y))}
		 \chi_{\bfh^0, \bfh^1}(x)}} \right|} \\
		 \ge \left( \frac{\delta^{2^{s+2}}}{2} \right)^{2^{s+1}}
		 = \frac{\delta^{2^{2s+3}}}{2^{2^{s+1}}},
	\end{align*}
	where
	\begin{align*}
		\tilde{\chi}_j = \begin{cases}
			\chi_j, & s = 0 \\
			1, & s \ge 1.
		\end{cases}
	\end{align*}
	After the substituting $x + a_l(y)$ for $x$, we have
	\begin{align*}
		\frac{1}{|G_1|^{2s}} \sum_{\bfh^0, \bfh^1 \in \Hil}{\left| \E_{x \in G_1}{\E_{y \in G_2}{\Delta_{\bfh^0 - \bfh^1}f_0(x)
		 \prod_{i=1}^{l-1}{\Delta_{\bfh^0 - \bfh^1}f_i(x + a_i(y))} \prod_{j=l+1}^m{\tilde{\chi}_j(a_j(y))}
		 \chi_{\bfh^0, \bfh^1}(x+a_l(y))}} \right|}
		 \ge \frac{\delta^{2^{2s+3}}}{2^{2^{s+1}}}.
	\end{align*}
	By the Cauchy--Schwarz inequality and 1-boundedness of $f_0$, we deduce that
	\begin{align*}
		\frac{1}{|G_1|^{2s}} \sum_{\bfh^0, \bfh^1 \in \Hil}{\E_{x \in G_1}{\left| \E_{y \in G_2}
		{\prod_{i=1}^{l-1}{\Delta_{\bfh^0 - \bfh^1}f_i(x + a_i(y))} \prod_{j=l+1}^m{\tilde{\chi}_j(a_j(y))}
		 \chi_{\bfh^0, \bfh^1}(a_l(y))} \right|^2}} \ge \frac{\delta^{2^{2s+4}}}{2^{2^{s+2}}}.
	\end{align*}
	By the hypothesis, for each $\bfh^0, \bfh^1 \in \Hil$,
	\begin{align*}
		\E_{x \in G_1}{\left| \E_{y \in G_2}{\prod_{i=1}^{l-1}{\Delta_{\bfh^0 - \bfh^1}f_i(x + a_i(y))}
		 \prod_{j=l+1}^m{\tilde{\chi}_j(a_j(y))} \chi_{\bfh^0, \bfh^1}(a_l(y))}
		 - \ind_{\tilde{\chi}=1} \ind_{\chi_{\bfh^0, \bfh^1} = 1}
		 \prod_{i=1}^{l-1}{\E_{z \in G_1}{\Delta_{\bfh^0 - \bfh^1} f_i(z)}} \right|^2}
		 \le \eps^2.
	\end{align*}
	It follows that
	\begin{align*}
		\left| \left\{ (\bfh^0, \bfh^1) \in \Hil^2 : \chi_{\bfh^0, \bfh^1} = 1 \right\} \right|
		 \ge \left( \frac{\delta^{2^{2s+4}}}{2^{2^{s+2}}} - \eps^2 \right) \cdot |G_1|^{2s},
	\end{align*}
	so there exists $\bfh^1 \in \Hil$ such that the set
	\begin{align*}
		\Hil' = \left\{ \bfh^0 \in \Hil : \chi_{\bfh^0, \bfh^1} = 1 \right\}
	\end{align*}
	has cardinality
	\begin{align*}
		\left| \Hil' \right|
		 \ge \left( \frac{\delta^{2^{2s+4}}}{2^{2^{s+2}}} - \eps^2 \right) \cdot |G_1|^s.
	\end{align*}
	Now, for $\bfh^0 \in \Hil'$, we have
	\begin{align} \label{eq: low rank expression}
		\chi_{\bfh^0} = \chi_{\bfh^0} \overline{\chi_{\bfh^0, \bfh^1}}
		 = \prod_{\omega \in \{0,1\}^s \setminus \{\mathbf{0}\}}{\chi_{\bfh^{\omega}}}.
	\end{align}
	Since $\bfh^1$ is fixed, \eqref{eq: low rank expression} represents $\chi_{\bfh^0}$
	as a product of functions depending on at most $s-1$ of the variables $(h^0_1, \dots, h^0_s)$.
	Define $\chi'_{\bfh}$ by the expression on the right hand side of \eqref{eq: low rank expression}.
	We have $\chi_{\bfh} = \chi'_{\bfh}$ for $\bfh \in \Hil'$, and the map $\bfh \mapsto \chi'_{\bfh}$
	satisfies the hypothesis of Lemma \ref{lem: low rank} with $m = s$.
	Thus, by Lemma \ref{lem: low rank},
	\begin{align*}
		\norm{U^{s+1}(G_1)}{\tilde{f}_l} \ge \E_{\bfh \in G_1^s}{\left| \E_{x \in G_1}{\Delta_{\bfh} f(x) \chi'_{\bfh}(x)} \right|}
		 \ge \frac{1}{|G_1|^s} \sum_{\bfh \in \Hil'}{\left| \E_{x \in G_1}{\Delta_{\bfh} f(x) \chi_{\bfh}(x)} \right|}
		 \ge \left( \frac{\delta^{2^{2s+4}}}{2^{2^{s+2}}} - \eps^2 \right) \cdot \frac{\delta^{2^{s+2}}}{2}.
	\end{align*}
\end{proof}

\begin{cor} \label{cor: U^s, U^1 bound}
	Let $G_1$ and $G_2$ be finite abelian groups, and let $a_1, \dots, a_m : G_2 \to G_1$.
	Let $l \in \{1, \dots, m\}$.
	Suppose that for any $f_1, \dots, f_{l-1} : G_1 \to \D$ and any $\chi_l, \dots, \chi_m \in \hat{G}_1$,
	\begin{align*}
		\norm{L^2(G_1)}{\E_{y \in G_2}{\prod_{i=1}^{l-1}{f_i(x + a_i(y))} \prod_{j=l}^m{\chi_j(a_j(y))}}
		 - \ind_{\chi=1} \prod_{i=1}^{l-1}{\E_{z \in G_1}{f_i(z)}}} \le \eps.
	\end{align*}
	
	Let $f_0, f_1, \dots, f_l : G_1 \to \D$ and $\chi_{l+1}, \dots, \chi_m \in \hat{G}_1$, and put
	\begin{align*}
		\tilde{f}_l(x) = \E_{y \in G_2}{f_0(x - a_l(y)) \prod_{i=1}^{l-1}{f_i(x + (a_i - a_l)(y))} \prod_{j=l+1}^m{\chi_j((a_j - a_l)(y))}}.
	\end{align*}
	For any $s \in \N$, there exist $m, n \in \N$ (depending only on $s$) such that
	\begin{align} \label{eq: U^s, U^1 bound}
		\norm{U^s(G_1)}{\tilde{f}_l} \ll_s \norm{U^1(G_1)}{\tilde{f}_l}^{1/m} + \eps^{1/n}.
	\end{align}
\end{cor}
\begin{proof}
	For $s = 1$, there is nothing to show.
	Suppose \eqref{eq: U^s, U^1 bound} holds for some $s \ge 1$.
	By Lemma \ref{lem: degree lowering},
	\begin{align*}
		\norm{U^s(G_1)}{\tilde{f}_l}
		 & \ge \left( \frac{\norm{U^{s+1}(G_1)}{\tilde{f}_l}^{4^{s+1}}}{2^{2^{s+1}}} - \eps^2 \right)
		 \frac{\norm{U^{s+1}(G_1)}{\tilde{f}_l}^{2^{s+1}}}{2} \\
		 & \gg_s \left( \norm{U^{s+1}(G_1)}{\tilde{f}_l} - 2^{1/2^{s+1}} \eps^{1/2^{2s+1}} \right)^{4^{s+1}}
		 \norm{U^{s+1}(G_1)}{\tilde{f}_l}^{2^{s+1}} \\
		 & \ge \left( \norm{U^{s+1}(G_1)}{\tilde{f}_l} - 2^{1/2^{s+1}} \eps^{1/2^{2s+1}} \right)^{4^{s+1} + 2^{s+1}}.
	\end{align*}
	Rearranging, we have
	\begin{align*}
		\norm{U^{s+1}(G_1)}{\tilde{f}_l}
		 \ll_s \norm{U^s(G_1)}{\tilde{f}_l}^{1/m'} + \eps^{1/n'}
	\end{align*}
	for
	\begin{align*}
		m' = 4^{s+1} + 2^{s+1} \qquad \text{and} \qquad n' = 2^{2s+1}.
	\end{align*}
	By the induction hypothesis,
	\begin{align*}
		\norm{U^s(G_1)}{\tilde{f}_l}^{1/m'} \ll_s \left( \norm{U^1(G_1)}{\tilde{f}_l}^{1/m} + \eps^{1/n} \right)^{1/m'}
		 \le \norm{U^1(G_1)}{\tilde{f}_l}^{1/(mm')} + \eps^{1/(nm')}.
	\end{align*}
	Thus,
	\begin{align*}
		\norm{U^{s+1}(G_1)}{\tilde{f}_l} \ll \norm{U^1(G_1)}{\tilde{f}_l}^{1/(mm')} + \eps^{1/n''},
	\end{align*}
	where $n'' = \max\{nm', n'\}$.
\end{proof}


\subsection{Proof of joint ergodicity criterion} \label{sec: proof}

We now carry out the argument outlined in Subsection \ref{sec: outline}.

\begin{proof}[Proof of Proposition \ref{prop: inductive version}]
	Suppose $l = 0$.
	We want to show: for any $\chi_1, \dots, \chi_m \in \hat{G}_1$,
	\begin{align} \label{eq: l=0}
		\left| \E_{y \in G_2}{\prod_{j=1}^m{\chi_j(a_j(y))}} - \ind_{\chi=1} \right| \le \eps_0
	\end{align}
	with $\eps_0 = O \left( \delta_1^{\gamma_{0,1}} + \delta_2^{\gamma_{0,2}} \right)$.
	If $\chi_1 = \dots = \chi_m = 1$, then the left hand side of \eqref{eq: l=0} is equal to 0.
	If, on the other hand, at least one of the characters $\chi_1, \dots, \chi_m$ is nontrivial,
	\eqref{eq: l=0} reduces to property (ii) with $\eps_0 = \delta_2$. \\
	
	Now suppose $l \in \{1, \dots, m\}$ and Proposition \ref{prop: inductive version} holds for $l-1$.
	Let $f_1, \dots, f_l : G_1 \to \D$ and $\chi_{l+1}, \dots, \chi_m \in \hat{G}_1$.
	We want to show
	\begin{align} \label{eq: l inequality}
		\norm{L^2(G_1)}{\E_{y \in G_2}{\prod_{i=1}^l{f_i(x + a_i(y))} \prod_{j=l+1}^m{\chi_j(a_j(y))}}
		 - \ind_{\chi=1} \prod_{i=1}^l{\E_{z \in G_1}{f_i(z)}}} \le \eps_l
	\end{align}
	with $\eps_l = O \left( \delta_1^{\gamma_{l,1}} + \delta_2^{\gamma_{l,2}} \right)$.
	Suppose \eqref{eq: l inequality} fails.
	Then for $f'_l = f_l - \E_{x \in G_1}{f_l(x)}$, we have
	\begin{align*}
		\norm{L^2(G_1)}{\E_{y \in G_2}{\prod_{i=1}^{l-1}{f_i(x + a_i(y))} f'_l(x + a_l(y)) \prod_{j=l+1}^m{\chi_j(a_j(y))}}}
		 \ge \eps_l - \eps_{l-1},
	\end{align*}
	where $\eps_{l-1} = O\left( \delta_1^{\gamma_{l-1},1} + \delta_2^{\gamma_{l-2},2} \right)$
	is the bound coming from the induction hypothesis.
	Combining with Lemma \ref{lem: multilinear}, we may assume that each of the functions $f_1, \dots, f_{l-1}$
	is either constant or has mean zero, $f_l$ has mean zero, and
	\begin{align}
		\norm{L^2(G_1)}{\E_{y \in G_2}{\prod_{i=1}^l{f_i(x + a_i(y))} \prod_{j=l+1}^m{\chi_j(a_j(y))}}}
		 \ge \frac{\eps_l - \eps_{l-1}}{2^{2l-1}}.
	\end{align}
	(Note that the extra factor of $2^l$ when compared with the statement of Lemma \ref{lem: multilinear}
	comes from dividing the modified functions by $2$ in order to preserve 1-boundedness of the functions $f_1, \dots, f_l$.)
	
	By Lemma \ref{lem: dual function}, there exists $f_0 : G_1 \to \D$ such that the function
	\begin{align*}
		\tilde{f}_l(x) = \E_{y \in G_2}{f_0(x - a_l(y)) \prod_{i=1}^{l-1}{f_i(x + (a_i - a_l)(y))} \prod_{j=l+1}^m{\chi_j((a_j - a_l)(y))}}
	\end{align*}
	satisfies
	\begin{align} \label{eq: U^1 small}
		\norm{U^1(G_1)}{\tilde{f}_1} \le \eps_{l-1}
	\end{align}
	and
	\begin{align*}
		\norm{L^2(G_1)}{\E_{y \in G_2}{\prod_{i=1}^{l-1}{f_i(x + a_i(y))} \tilde{f}_l(x + a_l(y)) \prod_{j=l+1}^m{\chi_j(a_j(y))}}}
		 \ge \frac{\delta^3}{4}
	\end{align*}
	with
	\begin{align*}
		\delta = \frac{\eps_l - \eps_{l-1}}{2^{2l-1}}.
	\end{align*}
	By property (i), it follows that
	\begin{align} \label{eq: U^s large}
		\norm{U^s(G_1)}{\tilde{f}_l} \ge C^{-1/\alpha} \left( \frac{\delta^3}{4} - \delta_1 \right)^{1/\alpha}.
	\end{align}
	On the other hand, by Corollary \ref{cor: U^s, U^1 bound}, there are $m, n \in \N$ such that
	\begin{align*}
		\norm{U^s(G_1)}{\tilde{f}_l} \ll_s \norm{U^1(G_1)}{\tilde{f}_l}^{1/m} + \eps_{l-1}^{1/n}.
	\end{align*}
	Combining with \eqref{eq: U^1 small} and \eqref{eq: U^s large}, we have
	\begin{align*}
		\left( \delta^3 - 4 \delta^1 \right)^{1/\alpha} \ll_{C, s, \alpha} \eps_{l-1}^{1/m} + \eps_{l-1}^{1/n} \ll \eps_{l-1}^{1/r},
	\end{align*}
	where $r = \max\{m,n\}$.
	Thus,
	\begin{align*}
		\eps_l \ll \eps_{l-1} + \delta \ll \eps_{l-1}^{\alpha/3r} + \delta_1^{1/3}.
	\end{align*}
	By induction, it follows that
	\begin{align*}
		\eps_l \ll \delta_1^{\gamma_{l,1}} + \delta_2^{\gamma_{l,2}}
	\end{align*}
	for
	\begin{align*}
		\gamma_{l,1} = \min\left\{ \frac{\alpha}{3r} \gamma_{l-1,1}, \frac{1}{3} \right\} \qquad \text{and} \qquad
		\gamma_{l,2} = \frac{\alpha}{3r} \gamma_{l-1,2}.
	\end{align*}
\end{proof}

\begin{rem}
	It follows from the proof that in the conclusion of Theorem \ref{thm: joint ergodicity}, we can take
	\begin{align*}
		\gamma_1 = \frac{1}{3} \left( \frac{\alpha}{3r} \right)^{s-2} \qquad \text{and} \qquad
		\gamma_2 = \left( \frac{\alpha}{3r} \right)^{s-1}
	\end{align*}
	with $r$ depending exponentially on $s$.
\end{rem}


\section{Gowers norm estimates} \label{sec: Gowers norm estimates}

Will we prove Theorem \ref{thm: asymptotic joint ergodicity} using the joint ergodicity criterion
(Theorem \ref{thm: joint ergodicity}).
As a first step, we therefore must control averages of the form
\begin{align*}
	\E_{y \in \F_q}{\prod_{i=1}^m{f_i(x+P_i(y))}}
\end{align*}
by an appropriate Gowers norm.

\begin{prop} \label{prop: Gowers norm control}
	Let $P_1, \dots, P_m \in (\F_p[t])[y]$ be nonconstant, essentially distinct polynomials.
	There exist $C_1, C_2 > 0$, $s \in \N$, and $\alpha, \beta \in (0,1]$ such that
	for any $q = p^k$ and any $f_1, \dots, f_m : \F_q \to \D$,
	\begin{align*}
		\norm{L^2(\F_q)}{\E_{y \in \F_q}{\prod_{i=1}^m{f_i(x+P_i(y))}}}
		 \le C_1 \min_{1 \le i \le m}{\norm{U^s(\F_q)}{f_i}^\alpha} + C_2 q^{-\beta}.
	\end{align*}
\end{prop}

The key estimates in the proof of Proposition \ref{prop: Gowers norm control} are given by the next two lemmas.

\begin{lem} \label{lem: finite index Gowers norm}
	Let $G$ be a finite abelian group, and let $H \le G$ be a subgroup.
	For any $f : G \to \D$ and any $s \in \N$,
	\begin{align*}
		\E_{h \in H}{\norm{U^s(G)}{\Delta_h f}^{2^s}} \le [G:H] \cdot \norm{U^{s+1}(G)}{f}^{2^{s+1}}.
	\end{align*}
\end{lem}
\begin{proof}
	Suppose $s = 1$.
	Expanding the left hand side and taking a Fourier transform, we have
	\begin{align*}
		\E_{h \in H}{\norm{U^1(G)}{\Delta_h f}^2}
		 & = \E_{h \in H}{\E_{v \in G}{\E_{u \in G}{f(u+v+h) \overline{f(u+v)} \overline{f(u+h)} f(u)}}} \\
		 & = \sum_{\chi_1, \chi_2, \chi_3, \chi_4 \in \hat{G}}
		 {\hat{f}(\chi_1) \overline{\hat{f}(\chi_2} \overline{\hat{f}(\chi_3)} \hat{f}(\chi_4)
		 \E_{h \in H}{(\chi_1 \overline{\chi}_2)(h)} \E_{v \in G}{(\chi_1 \overline{\chi_3})(v)}
		 \E_{u \in G}{(\chi_1 \overline{\chi}_2 \overline{\chi}_3 \chi_4)(u)}} \\
		 & = \sum_{\chi \in \hat{G}}{\left( \sum_{\lambda \in H^{\perp}}{\left| \hat{f}(\chi \lambda) \right|^2} \right)
		 \left| \hat{f}(\chi) \right|^2} \\
		 & \le \left( \sum_{\chi \in \hat{G}}{\left| \hat{f}(\chi) \right|^4} \right)^{1/2}
		 \left( \sum_{\chi \in \hat{G}}{\left( \sum_{\lambda \in H^{\perp}}
		 {\left| \hat{f}(\chi \lambda) \right|^2} \right)^2} \right)^{1/2} \\
		 & \le \left| H^{\perp} \right| \sum_{\chi \in \hat{G}}{\left| \hat{f}(\chi) \right|^4} \\
		 & = [G : H] \cdot \norm{U^2(G)}{f}^4.
	\end{align*} \\
	
	Now suppose $s \ge 2$.
	Then applying the $s = 1$ case above, we have
	\begin{align*}
		\E_{h \in H}{\norm{U^s(G)}{\Delta_h f}^{2^s}}
		 & = \E_{h \in H}{\E_{\mathbf{v} \in G^{s-1}}{\norm{U^1(G)}{\Delta_{\mathbf{v}} \Delta_h f}^2}} \\
		 & = \E_{\mathbf{v} \in G^{s-1}}{\left( \E_{h \in H}{\norm{U^1(G)}{\Delta_h \left( \Delta_{\mathbf{v}} f \right)}^2} \right)} \\
		 & = [G:H] \cdot \E_{\mathbf{v} \in G^{s-1}}{\norm{U^2(G)}{\Delta_{\mathbf{v}} f}^4} \\
		 & = [G:H] \cdot \norm{U^{s+1}(G)}{f}^{2^{s+1}}.
	\end{align*}
\end{proof}

\begin{lem} \label{lem: vdC bound}
	Let $P_1, \dots, P_m \in (\F_p[t])[y]$ be nonconstant, essentially distinct polynomials.
	Let $q = p^k$.
	Let $f_1, \dots, f_m : \F_q \to \D$.
	For $u \in \F_q$, let
	\begin{align*}
		\P_u = \{P_i(y) : 1 \le i \le m\} \cup \{P_i(y+u) : \ddeg{P_i} > 1\} = \{P_{u,1}, \dots, P_{u,l}\},
	\end{align*}
	and define
	\begin{align*}
		f_{u,j}(x) = \begin{cases}
			f_i(x), & \text{if}~P_{u,j}(y) = P_i(y+u)~\text{and}~\ddeg{P_i} > 1; \\
			\overline{f_i(x)}, & \text{if}~P_{u,j}(y) = P_i(y)~\text{and}~\ddeg{P_i} > 1; \\
			f_i(x + P_i(u) - P_i(0)) \overline{f_i(x)} & \text{if}~P_{u,j}(y) = P_i(y)~\text{and}~\ddeg{P_i} = 1
		\end{cases}
	\end{align*}
	for $j = 2, \dots, l$.
	Then
	\begin{align*}
		\norm{L^2(\F_q)}{\E_{y \in \F_q}{\prod_{i=1}^m{f_i(x+P_i(y))}}}^2
		 \le \E_{u \in \F_q}{\norm{L^2(\F_q)}{\E_{y \in \F_q}{\prod_{j=2}^l{f_{u,j} \left( x + (P_{u,j} - P_{u,1})(y) \right)}}}}.
	\end{align*}
\end{lem}
\begin{proof}
	Expanding the left hand side,
	\begin{align*}
		\norm{L^2(\F_q)}{\E_y{\prod_{i=1}^m{f_i(x+P_i(y))}}}^2
		 = \E_{x,y,u}{\prod_{i=1}^m{f_i(x + P_i(y+u))} \prod_{j=1}^m{\overline{f_j(x + P_j(y))}}}.
	\end{align*}
	Using the definition of $f_{u,j}$ and then shifting by $P_{u,1}(y)$, this is equal to
	\begin{align*}
		\E_{x,y,u}{\prod_{j=1}^l{f_{u,j}(x + P_{u,j}(y))}}
		 & = \E_u{\E_x{\left( f_{u,1}(x) \cdot \E_y{\prod_{j=2}^l{f_{u,j} \left( x + (P_{u,j} - P_{u,1})(y) \right)}} \right)}},
	\end{align*}
	and the desired inequality follows by an application of the Cauchy--Schwarz inequality together with the bound
	$\norm{L^2(\F_q)}{f_{u,1}} \le \norm{L^{\infty}(\F_q)}{f_{u,1}} \le 1$.
\end{proof}

We prove Proposition \ref{prop: Gowers norm control} using a version of PET induction.
The appropriate version of PET induction in our context depends on the derivational degrees
of the polynomials $P_1, \dots, P_m$ rather than their degrees.
We therefore need a notion of leading term that is different from the usual meaning.
In particular, for a polynomial $P(y) \in (\F_p[t])[y]$, we write $P(y) = P_k(y) + P_{<k}(y)$,
where each monomial in $P_k(y)$ is of derivational degree $k = \ddeg{P}$
and $P_{<k}$ has derivational degree $\ddeg{P_{<k}} < k$.
We then call $P_k$ the \emph{derivational leading term} of $P$.

To any family of polynomials $\P = \{P_1, \dots, P_m\} \subseteq (\F_p[t])[y]$,
we assign a \emph{weight} $w = w(\P) \in \N_0^{\N}$,
where $w_i$ is the number of distinct derivational leading terms of deriviational degree $i$.
We order weight vectors using the anti-lexicographic ordering.
That is, $w < w'$ if there is some $d \in \N$ such that $w_d < w'_d$ and $w_i = w'_i$ for $i > d$.
We will induct on the weight of the family $\P$ in order to prove Proposition \ref{prop: Gowers norm control}.

\begin{proof}[Proof of Proposition \ref{prop: Gowers norm control}]
	Since the statement of Proposition \ref{prop: Gowers norm control} is symmetric in the polynomials $P_1, \dots, P_m$,
	it suffices to prove
	\begin{align} \label{eq: f1 bound}
		\norm{L^2(\F_q)}{\E_{y \in \F_q}{\prod_{i=1}^m{f_i(x+P_i(y))}}}
		 \le C_1 \norm{U^s(\F_q)}{f_1}^{\alpha} + C_2 q^{-\beta}
	\end{align}
	for some $C_1, C_2 > 0$, $s \in \N$ and $\alpha, \beta \in (0,1]$.
	
	For a family of polynomials $P_1, \dots, P_m$, let $\ddeg{(P_1, \dots, P_m)} = \max_{1 \le i \le m}{\ddeg{P_i}}$.
	Following Leibman \cite{leibman}, we say that a family of essentially distinct polynomials is \emph{standard}
	if $\ddeg{P_1} = \ddeg{(P_1, \dots, P_m)}$.
	We will first establish the inequality \eqref{eq: f1 bound} for standard families. \\
	
	Suppose $P(y) \in (\F_p[t])[y]$ is an additive polynomial, say $P(y) = \sum_{j=0}^N{a_j p^j}$.
	Let $H = \left\{ P(y) : y \in \F_q \right\} \subseteq \F_q$.
	Since $P$ is a polynomial of degree $p^N$, it has at most $p^N$ roots in $\F_q$.
	Hence, $[\F_q : H] \le p^N$.
	Expanding a function $f : \F_q \to \D$ as a Fourier series, one can show
	\begin{align} \label{eq: annihilator identity}
		\norm{L^2(\F_q)}{\E_{y \in \F_q}{f \left( x + P(y) \right)}}^2
		 = \sum_{\chi \in H^{\perp}}{\left| \hat{f}(\chi) \right|^2},
	\end{align}
	where $H^{\perp} = \left\{ \chi \in \hat{\F_q} : \chi|_H = 1 \right\}$.
	By \cite[Section 2.1]{rudin}, we have an isomorphism $H^{\perp} \cong \hat{\F_q/H}$, so $|H^{\perp}| = [\F_q : H]$.
	Then by the Cauchy--Schwarz inequality,
	\begin{align} \label{eq: finite index bound}
		\left( \sum_{\chi \in H^{\perp}}{\left| \hat{f}(\chi) \right|^2} \right)^2
		 \le |H^{\perp}| \sum_{\chi \in \F_q}{\left| \hat{f}(\chi) \right|^4}
		 \le p^N \norm{U^2(\F_q)}{f}^4.
	\end{align}
	Combining \eqref{eq: annihilator identity} and \eqref{eq: finite index bound},
	\begin{align*}
		\norm{L^2(\F_q)}{\E_{y \in \F_q}{f \left( x + P(y) \right)}} \le p^{N/4} \norm{U^2(\F_q)}{f}.
	\end{align*}
	This proves the base case corresponding to the weight vector $(1, 0, \dots, 0)$. \\
	
	Suppose \eqref{eq: f1 bound} holds for all standard families with weight $ < w(\P)$.
	Let $i_0 \in \{2, \dots, m\}$ such that $P_{i_0}$ has minimal degree
	and, if possible, such that the derivational leading term of $P_{i_0}$ is different from the derivational leading term of $P_1$.
	We apply Lemma \ref{lem: vdC bound} with $P_{u,1} = P_{i_0}$ and $P_{u,2}(y) = P_1(y+u)$ to obtain the estimate
	\begin{align} \label{eq: vdC bound}
		\norm{L^2(\F_q)}{\E_{y \in \F_q}{\prod_{i=1}^m{f_i(x+P_i(y))}}}^2
		 \le \E_{u \in \F_q}{\norm{L^2(\F_q)}{\E_{y \in \F_q}{\prod_{j=2}^l{f_{u,j} \left( x + (P_{u,j} - P_{u,1})(y) \right)}}}}.
	\end{align}
	(In the case that all of the polynomials $P_1, \dots, P_m$ have $\ddeg{P_i}=1$, we take $P_{u,2} = P_1$ instead,
	and the remainder of the argument works out the same.)
	
	Let
	\begin{align*}
		\P'_u = \left\{ P_{u,2} - P_{u,1}, \dots, P_{u,l} - P_{u,1} \right\}
	\end{align*}
	be the family of polynomials appearing on the right hand side of \eqref{eq: vdC bound}.
	We claim $w(\P'_u) < w(\P)$.
	Fix $i \in \{2, \dots, l\}$.
	If $\ddeg{P_{u,i}} > \ddeg{P_{u,1}}$, then the derivational leading term of $P_{u,i} - P_{u,1}$
	is the same as the derivational leading term of $P_{u,i}$.
	If $\ddeg{P_{u,i}} = \ddeg{P_{u,1}}$ but $P_{u,i}$ and $P_{u,1}$ have different derivational leading terms,
	then the derivational leading term of $P_{u,i} - P_{u,1}$ is equal to the difference
	of the derivational leading terms of $P_{u,i}$ and $P_{u,1}$.
	Finally, if $P_{u_i}$ and $P_{u,1}$ have the same derivational leading term,
	then $\ddeg{(P_{u,i} - P_{u,1})} < \ddeg{P_{u,i}}$.
	Therefore, the $j$th coordinate of the weight vector $w(\P'_u)$ agrees with
	the $j$th coordinate of $w(\P_u)$ for $j > \ddeg{P_{i_0}}$ and is one smaller for $j = \ddeg{P_{i_0}}$.
	Noting that $w(\P_u) = w(\P)$, we have $w(\P'_u) < w(\P)$ as claimed.
	
	The family $\P'_u$ is also standard:
	$\ddeg{P_{u,2}} = \ddeg{P_1} = \ddeg{\P'_u}$, unless all of the polynomials $P_1, \dots, P_m$
	have the same derivational leading term, in which case $\ddeg{P_{u,2}} = \ddeg{P_1} - 1 = \ddeg{\P'_u}$.
	
	Therefore, by the induction hypothesis, for any $u \ne 0$,
	\begin{align*}
		\norm{L^2(\F1)}{\E_y{\prod_{j=2}^l{f_{u,j} \left( x + (P_{u,j} - P_{u,1})(y) \right)}}}
		 \le C_1 \norm{U^s(\F_q)}{f_{u,2}}^{\alpha} + C_2 q^{-\beta}
	\end{align*}
	for some $C_1, C_2 > 0$, $s \in \N$ and $\alpha, \beta \in (0,1]$.
	Hence,
	\begin{align*}
		\norm{L^2(\F_q)}{\E_{y \in \F_q}{\prod_{i=1}^m{f_i(x+P_i(y))}}}^2
		 & \le \frac{1}{q} + \frac{1}{q} \sum_{u \ne 0}
		 {\left( C_1 \norm{U^s(\F_q)}{f_{u,2}}^{\alpha} + C_2 q^{-\beta} \right)} \\
		 & \le C_1 \left( \E_u{\norm{U^s(\F_q)}{f_{u,2}}} \right)^{\alpha} + 2 C_2 q^{-\beta}.
	\end{align*}
	
	It remains to bound $\E_u{\norm{U^s(\F_q)}{f_{u,2}}}$ in terms of a Gowers norm of $f_1$.
	If $\ddeg{P_1} > 1$ and $P_{u,2}(y) = P_1(y+u)$, then $f_{u,2} = f_1$,
	so in particular, $\norm{U^s(\F_q)}{f_{u,2}} = \norm{U^s(\F_q)}{f_1}$.
	Therefore,
	\begin{align*}
		\norm{L^2(\F_q)}{\E_{y \in \F_q}{\prod_{i=1}^m{f_i(x+P_i(y))}}}
		 \le C_1^{1/2} \norm{U^s(\F_q)}{f_1}^{\alpha/2} + \sqrt{2} C_2^{1/2} q^{-\beta/2}.
	\end{align*}
	Suppose instead that $\ddeg{P_1} = 1$ and $P_{u,2} = P_1$.
	Write $P_1(y) = \eta(y) + c$ with $\eta(y) = \sum_{i=0}^N{a_iy^{p^i}}$ an additive polynomial of degree $p^N$.
	Then $f_{u,2}(x) = f_1(x + \eta(u)) \overline{f_1(x)} = \Delta_{\eta(u)} f_1(x)$, so
	\begin{align*}
		\E_{u \in \F_q}{\norm{U^s(\F_q)}{f_{u,2}}} = \E_{h \in H}{\norm{U^s(\F_q)}{\Delta_h f_1}}
		 \le \left( \E_{h \in H}{\norm{U^s(\F_q)}{\Delta_h f_1}^{2^s}} \right)^{1/2^s},
	\end{align*}
	where $H = \eta(\F_q)$.
	Note that $[\F_q : H] = \left| \left\{ y \in \F_q : \eta(y) = 0 \right\} \right| \le p^N$.
	Hence, by Lemma \ref{lem: finite index Gowers norm},
	\begin{align*}
		\E_{u \in \F_q}{\norm{U^s(\F_q)}{f_{u,2}}}
		 \le \left( p^N \cdot \norm{U^{s+1}(\F_q)}{f_1}^{2^{s+1}} \right)^{1/2^s}
		 = p^{N/2^s} \norm{U^{s+1}(\F_q)}{f_1}^2.
	\end{align*}
	Thus,
	\begin{align*}
		\norm{L^2(\F_q)}{\E_{y \in \F_q}{\prod_{i=1}^m{f_i(x+P_i(y))}}}
		 \le p^{\alpha N/2^{s+1}} C_1^{1/2} \norm{U^{s+1}(\F_q)}{f_1}^{\alpha} + \sqrt{2} C_2^{1/2} q^{-\beta/2}.
	\end{align*}
	
	We have shown that for any standard family $\P = \{P_1, \dots, P_m\}$,
	there exist $C_1, C_2 > 0$, $s \in \N$, and $\alpha, \beta \in (0,1]$ such that
	\begin{align*}
		\norm{L^2(\F_q)}{\E_{y \in \F_q}{\prod_{i=1}^m{f_i(x+P_i(y))}}}
		 \le C_1 \norm{U^s(\F_q)}{f_1}^{\alpha} + C_2 q^{-\beta/2}.
	\end{align*} \\
	
	Now suppose $\P = \{P_1, \dots, P_m\}$ is not standard.
	Let $i_0 \in \{2, \dots, m\}$ such that $\ddeg{P_{i_0}} = \ddeg{\P} > 1$.
	By Lemma \ref{lem: vdC bound} with $P_{u,1} = P_{i_0}$ and $P_{u,2} = P_1$, we have
	\begin{align*}
		\norm{L^2(\F_q)}{\E_{y \in \F_q}{\prod_{i=1}^m{f_i(x+P_i(y))}}}^2
		 \le \E_{u \in \F_q}{\norm{L^2(\F_q)}{\E_{y \in \F_q}{\prod_{j=2}^l{f_{u,j} \left( x + (P_{u,j} - P_{u,1})(y) \right)}}}}.
	\end{align*}
	The family
	\begin{align*}
		\P'_u = \left\{ P_{u,2} - P_{u,1}, \dots, P_{u,l} - P_{u,1} \right\}
	\end{align*}
	is standard since $\ddeg{(P_{u,2} - P_{u,1})} = \ddeg{(P_1 - P_{i_0})} = \ddeg{P_{i_0}} = \ddeg{\P} \ge \ddeg{\P'_u}$.
	Thus,
	\begin{align*}
		\norm{L^2(\F_q)}{\E_{y}{\prod_{j=2}^l{f_{u,j} \left( x + (P_{u,j} - P_{u,1})(y) \right)}}}
		 \le C_1 \norm{U^s(\F_q)}{f_{u,2}}^{\alpha} + C_2 q^{-\beta}
	\end{align*}
	for some $C_1, C_2 > 0$, $s \in \N$ and $\alpha, \beta \in (0,1]$.
	Taking an average over $u \in \F_q$ and arguing as above,
	\begin{align*}
		\norm{L^2(\F_q)}{\E_{y \in \F_q}{\prod_{i=1}^m{f_i(x+P_i(y))}}}
		 & \le C'_1 \norm{U^{s'}(\F_q)}{f_1}^{\alpha'} + C'_2 q^{-\beta'}.
	\end{align*}
\end{proof}


\section{Asymptotic joint ergodicity for polynomial sequences} \label{sec: joint ergdocity polynomials}

We are now set to prove Theorem \ref{thm: asymptotic joint ergodicity}.

\begin{proof}[Proof of Theorem \ref{thm: asymptotic joint ergodicity}]
	We will use Theorem \ref{thm: joint ergodicity}.
	First, by Proposition \ref{prop: Gowers norm control},
	there are constants $C_1, C_2 > 0$, $s \in \N$, and $\alpha, \beta_1 \in (0,1]$ such that
	for any $q = p^k$, any $l \in \{1, \dots, m\}$, any $f_1, \dots, f_l : \F_q \to \D$,
	and any $\chi_{l+1}, \dots, \chi_m \in \hat{\F}_q$,
	\begin{align*}
		\norm{L^2(\F_q)}{\E_{y \in \F_q}{\prod_{i=1}^l{f_i(x + P_i(y))} \prod_{j=l+1}^m{\chi_j(P_j(y))}}}
		 \le C_1 \norm{U^s(\F_q)}{f_l}^{\alpha} + C_2 q^{-\beta_1}.
	\end{align*}
	That is, property (i) in Theorem \ref{thm: joint ergodicity} is satisfied, with $\delta_1 = C_2 q^{-\beta_1}$. \\
	
	Next, by \cite[Theorem 1.10]{ab1},
	there exist $C_3 > 0$ and $\beta_2 \in (0,1]$ such that, for any linear combination $P(y) = \sum_{i=1}^m{s_iP_i(y)}$,
	any $q = p^k$, and any $f : \F_q \to \D$,
	\begin{align*}
		\norm{L^2(\F_q)}{\E_{y \in \F_q}{f(x+P(y))} - \E_{z \in H_q(s_1, \dots, s_m)}{f(x+a_0+z)}}
		 \le C_3 q^{-\beta_2} \norm{L^2(\F_q)}{f},
	\end{align*}
	where $a_0 = P(0)$ and $H_q(s_1, \dots, s_m)$ is the subgroup of $\F_q$
	generated by $\left\{ P(y) - a_0 : y \in \F_q \right\}$.
	Given $\chi_1, \dots, \chi_m \in \hat{\F}_q$, there exist $s_i \in \F_p[t]$ such that $\chi_i(x) = e(s_ix/Q)$,
	$Q(t) \in \F_p[t]^+$ irreducible with $|Q| = q$.
	Taking $f(x) = e(x/Q)$, we then have
	\begin{align} \label{eq: asymptotic distribution}
		\left| \E_{y \in \F_q}{\prod_{i=1}^m{\chi_i(P_i(y))}} - e(a_0/Q) \ind_{H_q(s_1, \dots, s_m)^{\perp}}(1) \right|
		 \le C_3 q^{-\beta_2}.
	\end{align}
	
	We claim $H_q(s_1, \dots, s_m) = \F_q$ for all sufficiently large $q$ and all $s_1, \dots, s_m \in \F_q$ not all zero.
	Suppose not.
	Let
	\begin{align*}
		A = \left\{ (x_1, \dots, x_m) \in \left( \F_p((t^{-1}))/\F_p[t] \right)^m
		 : e \left( \sum_{i=1}^m{P_i(y)x_i} \right) = 1~\text{for all}~y \in \F_p[t] \right\}.
	\end{align*}
	Since $x \mapsto e(x)$ is a continuous homomorphism, the set $A$ is a closed subgroup.
	The assumption that $H_q(s_1, \dots, s_m) \ne \F_q$ for arbitrarily large $q$ and $s_1, \dots, s_m \in \F_q$ not all zero
	means that $A$ contains infinitely many rational points.
	Since $A$ is a compact group, it follows that $A$ is uncountable.
	In particular, $A$ contains a point $(\alpha_1, \dots, \alpha_m)$ with at least one $\alpha_i$ irrational.
	This contradicts the assumption that $\{P_1, \dots, P_m\}$ is good for irrational equidistribution,
	so the claim holds.
	
	By the claim, the inequality \eqref{eq: asymptotic distribution} simplifies to
	\begin{align*}
		\left| \E_{y \in \F_q}{\prod_{i=1}^m{\chi_i(P_i(y))}} \right| \le C_3 q^{-\beta_2}.
	\end{align*}
	Hence, property (ii) in Theorem \ref{thm: joint ergodicity} is satisfied, with $\delta_2 = C_3 q^{-\beta_2}$. \\
	
	By Theorem \ref{thm: joint ergodicity}, there exist $\gamma_1, \gamma_2 > 0$ such that
	\begin{align*}
		\norm{L^2(\F_q)}{\E_{y \in \F_q}{\prod_{i=1}^m{f_i(x + P_i(y))}} - \prod_{i=1}^m{\E_{z \in \F_q}{f_i(z)}}}
		 \ll_{P_1, \dots, P_m} (C_2 q^{-\beta_1})^{\gamma_1} + (C_3 q^{-\beta_2})^{\gamma_2}
	\end{align*}
	for any $f_1, \dots, f_m : G \to \D$.
	Taking $\gamma = \min\{\beta_1\gamma_1, \beta_2 \gamma_2\}$ gives the desired inequality:
	\begin{align*}
		\norm{L^2(\F_q)}{\E_{y \in \F_q}{\prod_{i=1}^m{f_i(x + P_i(y))}} - \prod_{i=1}^m{\E_{z \in \F_q}{f_i(z)}}}
		 \ll_{P_1, \dots, P_m} q^{-\gamma}.
	\end{align*}
\end{proof}


\section{Power saving bound} \label{sec: power saving}

We now deduce a power saving bound for the polynomial Szemer\'{e}di theorem over finite fields
when the family of polynomials is good for irrational equidistribution.
Recall the statement of Corollary \ref{cor: power saving} from the introduction:

{\renewcommand\footnote[1]{}\PowerSaving*}

\begin{proof}
	Let
	\begin{align*}
		N(A_0, \dots, A_m) = \left| \left\{ (x,y) \in \F_q^2 : x \in A_0, x + P_1(y) \in A_1, \dots, x + P_m(y) \in A_m \right\} \right|.
	\end{align*}
	For each $i \in \{0, \dots, m\}$, let $\alpha_i = q^{-1} |A_i| = \E_z{\ind_{A_i}(z)}$ be the density of the set $A_i$.
	By the Cauchy--Schwarz inequality,
	\begin{align*}
		\left| N(A_0, \dots, A_m) - q^{-(m-1)} \prod_{i=0}^m{|A_i|} \right|
		 & = q^2 \left| \E_x{\ind_{A_0}(x) \cdot \E_y{\prod_{i=1}^m{\ind_{A_i}(x + P_i(y))}}} - \prod_{i=0}^m{\alpha_i} \right| \\
		 & = q^2 \left| \E_x{\ind_{A_0}(x) \left( \E_y{\prod_{i=1}^m{\ind_{A_i}(x + P_i(y))}}
		 - \prod_{i=1}^m{\alpha_i} \right)} \right| \\
		 & \le q^2 \norm{L^2(\F_q)}{\ind_{A_0}}
		 \norm{L^2(\F_q)}{\E_y{\prod_{i=1}^m{\ind_{A_i}(x + P_i(y))}} - \prod_{i=1}^m{\alpha_i}}.
	\end{align*}
	Note that $\norm{L^2(\F_q)}{\ind_{A_0}} = \alpha_0^{1/2}$.
	Moreover, by Theorem \ref{thm: asymptotic joint ergodicity},
	\begin{align*}
		\norm{L^2(\F_q)}{\E_y{\prod_{i=1}^m{\ind_{A_i}(x + P_i(y))}} - \prod_{i=1}^m{\alpha_i}} \ll q^{-\gamma}.
	\end{align*}
	Thus,
	\begin{align*}
		\left| N(A_0, \dots, A_m) - q^{-(m-1)} \prod_{i=0}^m{|A_i|} \right|
		 \ll |A_0|^{1/2} q^{3/2 - \gamma}.
	\end{align*} \\
	
	Suppose $A \subseteq \F_q$ contains no nontrivial pattern $\{x, x + P_1(y), \dots, P_m(y)\}$.
	Then by the union bound,
	\begin{align*}
		N(A, \dots, A) & \le \left| \left\{ (x,y) \in \F_q^2 : x \in A, \left| \{0, P_1(y), \dots, P_m(y)\} \right| \le m \right\} \right| \\
		 & \le |A| \left( \sum_{i=1}^m{\left| \left\{ y \in \F_q : P_i(y) = 0 \right\} \right|}
		 + \sum_{1 \le i < j \le m}{\left| \left\{ y \in \F_q : (P_j - P_i)(y) = 0 \right\} \right|} \right) \\
		 & \le \left( m + \binom{m}{2} \right) d |A|,
	\end{align*}
	where $d = \max_{1 \le i \le m}{\deg{P_i}}$.
	Therefore, by the above,
	\begin{align*}
		q^{-(m-1)} |A|^{m+1} - D |A| \ll |A|^{1/2} q^{3/2 - \gamma},
	\end{align*}
	where we have put $D = \left( m + \binom{m}{2} \right) d = \binom{m+1}{2} d$.
	Multiplying by $|A|^{-1/2} q^{m-1}$ and keeping only the dominant term, we have
	\begin{align*}
		|A|^{m+1/2} \ll q^{m+1/2 - \gamma}.
	\end{align*}
\end{proof}


\section*{Acknowledgements}

The first author is supported by the National Science Foundation under Grant No. DMS-1926686.


\end{document}